\documentclass[11pt]{amsart}
\usepackage[a4paper,margin=1in]{geometry}
\usepackage{amsmath,amssymb,amsthm,mathtools,mathrsfs}
\usepackage{hyperref}
\usepackage{enumitem}
\usepackage{bm}
\usepackage{cite}
\usepackage{microtype}
\hypersetup{colorlinks=true,linkcolor=blue,citecolor=blue,urlcolor=blue}

\newtheorem{theorem}{Theorem}[section]
\newtheorem{proposition}[theorem]{Proposition}

\newtheorem{lemma}[theorem]{Lemma}
\newtheorem{remark}[theorem]{Remark}
\newtheorem{definition}[theorem]{Definition}
\numberwithin{equation}{section} \allowdisplaybreaks
\makeatletter

\newcommand{\Rmnum}[1]{\expandafter\@slowromancap\romannumeral #1@}
\makeatother

\begin{document}

\title[Structure of singular points for the logarithmic obstacle problem]{Geometric structure of singular free boundary points for the logarithmic obstacle problem}

\author{Lili Du}
\address{\textsc{Lili Du:}\newline
	Department of Mathematics, Sichuan University,
	Chengdu 610064, P.\,R.\,China.}
\email[L. Du]{dulili@scu.edu.cn}

\author{Xu Tang}
\address{\textsc{Xu Tang:}\newline
	School of Mathematical Sciences, Fudan University,
	Shanghai 200433, P.\,R.\,China.}
\email[X. Tang]{tangxu8988@163.com}

\author{Yi Zhou}
\address{\textsc{Yi Zhou$^{\ast}$:}\newline
	School of Mathematical Sciences, Shenzhen University,
	Shenzhen 518060, P.\,R.\,China.}
\email[Y. Zhou]{zhouyimath@163.com}

\thanks{$^{\ast}$Corresponding author: Yi Zhou (zhouyimath@163.com)}

\subjclass[2020]{35R35; 35J87; 35B44.}

\date{}

\begin{abstract}
In the previous work [Interfaces Free Bound., 19, 351--369, 2017],
de Queiroz and Shahgholian established the optimal $C^{1,\log}_{\mathrm{loc}}$
regularity of solutions for the obstacle problem with singular logarithmic forcing term
$$-\Delta u = \log u\,\chi_{\{u>0\}} \quad \text{in } \Omega,$$
where $\Omega\subset\mathbb{R}^d$ ($d\geq 2$) is a smooth bounded domain.
In our earlier work [arXiv:2408.08104, 2024], we proved the
$C^{1,\alpha}$ regularity of the free boundary $\Omega\cap\partial\{u>0\}$
near regular points. In this paper, we investigate the more delicate
structure of the \emph{singular} free boundary.
Since the nonlinearity $-\log u$ is singular near the
free boundary and destroys the scaling invariance, so that neither the
classical blow-up arguments nor the standard epiperimetric inequality
[Weiss, Invent.\ Math., 138, 23--50, 1999] apply directly; moreover, the
Weiss-type monotonicity formula requires a variable-parameter correction
that introduces non-integrable remainder terms into the energy estimates.
Motivated by Colombo--Spolaor--Velichkov [Geom.\ Funct.\ Anal., 28, 1029--1061, 2018], we develop a new \emph{log-epiperimetric
	inequality} for the modified Weiss energy, also proved by the direct method.
A key novelty is the introduction of an auxiliary correction term $T$ that
absorbs the non-integrable errors. As consequences, we establish a
logarithmic energy decay, uniqueness of blow-ups at singular points, and a $C^{1,\log}$-type geometric description of the singular strata. In dimension two, the logarithmic modulus improves to a H\"older modulus.

\medskip
\noindent\textbf{Keywords:} Free boundary;\; Obstacle problem;\; Singular point;\; Log-epiperimetric inequality.
\end{abstract}

\maketitle

\tableofcontents

\section{Introduction}
In this paper, we investigate the structure of the singular free boundary points for minimizers of the functional
\begin{equation}\label{functional u+}
	\mathcal{L}(u; \Omega):=\int_{\Omega}\left(\frac{|\nabla u|^2}{2}-u^+ (\log u-1)\right) \, dx,
\end{equation}
where $u^+=\max{\{0,u\}}$.  The minimization is performed over the admissible class$$\mathcal{K}_{\varphi} := \Big\{ u \in W^{1,2}(\Omega) : u = \varphi \text{ on } \partial \Omega \Big\},$$
where $\varphi \in H^1(\Omega) \cap L^{\infty}(\Omega)$ is a fixed non-negative boundary datum. The corresponding Euler-Lagrange equation is given by\begin{equation}\label{euler Lagrange}-\Delta u = \log u \, \chi_{\{u>0\}} \quad \text{in } \Omega,\end{equation}where $\chi_{\{u>0\}}$ denotes the characteristic function of the set $\{u>0\}$, and $\Omega \subset \mathbb{R}^d$ ($d \geq 2$) is a smooth bounded domain. We denote the free boundary by $\mathscr{F}(u) := \Omega \cap \partial \{u > 0\}$.

It has been established in \cite[Lemma 2.1]{qs17} that minimizers of $\mathcal{L}(u,\Omega)$ \eqref{functional u+} are always non-negative. Consequently, the minimization problem \eqref{functional u+} is equivalent to the following problem (see \cite[Theorem 1.4]{psu} for further details):
\begin{equation}\label{functional}
	\mathcal{L}_0(u;\Omega):=\int_{\Omega}\left(\frac{|\nabla u|^2}{2}+u( -\log u+1)\right) dx,
\end{equation}
restricted to the class $$\mathcal{K}_0:=\{u\in W^{1,2}(\Omega): u-\varphi\in W_0^{1,2}(\Omega),\  u\geq 0 \quad\text{in}\quad \Omega\}.$$
Despite the non-convexity of the functional \eqref{functional}, we study its minimizers, which satisfy the Euler-Lagrange equation \eqref{euler Lagrange}.

For convenience, throughout this paper we adopt the following notation:
\begin{align*}
	F(u)=u( -\log u+1),\quad\text{and}\quad	f(u):=F'(u)=-\log u.
\end{align*}

\subsection{Background of the obstacle problem}\ 

The study of obstacle-type problems originates from the classical obstacle problem, which physically describes the equilibrium position of an elastic membrane constrained to remain above a given obstacle. Let $u$ denote the displacement of the membrane; the region where the membrane contacts the obstacle is known as the contact set, and its boundary within the domain is the \textit{free boundary}, which is the primary object of mathematical interest. The mathematical model for the classical obstacle problem is given by the following elliptic equation:
\begin{align}\label{cop euler 2}
	\Delta u &= \chi_{\{u>0\}} \quad \text{in } \Omega, \\
	u &= \varphi \quad \text{on } \partial\Omega, \nonumber
\end{align}
where $\varphi \in H^1(\Omega) \cap L^{\infty}(\Omega)$ is a prescribed non-negative boundary condition. The solution to this equation corresponds to the minimizer of the energy functional
\begin{equation*}
	\int_{\Omega} \left( \frac{|\nabla u|^2}{2} + u \right) dx,
\end{equation*}
over the admissible class 
$\mathcal{K} := \{u \in W^{1,2}(\Omega) : u - \varphi \in W_0^{1,2}(\Omega), \, u \geq 0 \text{ in } \Omega\}.$

Regarding this classical obstacle problem, an abundance of results has been established. We shall categorize these findings into two main aspects: the regularity of the solution and the regularity of the free boundary. A cornerstone result in this field was achieved by Frehse \cite{f72}, who proved that the solution $u$ possesses optimal $C_{\rm loc}^{1,1}$ regularity. Rigorous proofs and extensions to general second-order linear elliptic operators can be found in \cite{ger73, bk74, psu, frro22}. 

Intuitively, since the Laplacian $\Delta u$ in \eqref{cop euler 2} jumps from $0$ to $1$ across the free boundary, one cannot expect the solution to be $C^2$ across the interface; thus, $C^{1,1}$ is indeed the best possible regularity.

The focus of this paper is the structure of the free boundary $\mathscr{F}(u) := \partial \{u>0\} \cap \Omega$. For issues regarding the intersection of the free boundary with the fixed boundary $\partial\Omega$, we refer the reader to \cite[Chapter 8]{psu}. Furthermore, while the right-hand side of \eqref{cop euler 2} is constant, the techniques developed for this case have been successfully extended to more general force terms $f(x)$ in \cite{cs15}.

The study of free boundary regularity was revolutionized by the landmark work of Caffarelli. In his seminal paper \cite{c77}, he established the \textit{Caffarelli's dichotomy theorem}, which classifies free boundary points into two distinct categories based on the local behavior of the contact set:
\begin{itemize}
	\item \textit{Regular Points}: Points where the contact set, after scaling, asymptotically behaves like a half-space. At such points, the free boundary is locally smooth.
	\item \textit{Singular Points}: Points where the contact set is ``thin" (having zero Lebesgue density) under scaling.
\end{itemize}

To rigorously characterize the geometry of $\mathscr{F}(u)$, Caffarelli introduced the technique of \textit{blow-up analysis}. By examining the limit of a scaling sequence
\begin{equation}\label{cop scaling}
	u_r(x) = \frac{u(x^0 + rx)}{r^2}, \quad r \to 0+.
\end{equation}
The choice of $r^2$ is dictated by the optimal growth estimate and \textit{non-degeneracy} \cite{c77}: there exist constants $C_1, C_2 > 0$ such that for any $x^0 \in \mathscr{F}(u)$ and sufficiently small $r$,
$$C_1 r^2 \leq \sup_{B_r(x^0)} u \leq C_2 r^2.$$

This quadratic growth guarantees that blow-up limits are nontrivial 2-homogeneous global solutions of the classical obstacle problem.
The classification of these limits is given by:
\begin{itemize}
	\item \textit{Half-space solutions}: $u_0(x) = \frac{1}{2}(x \cdot e)_+^2$ for some unit vector $e$, corresponding to regular points.
	\item \textit{Polynomial solutions}: $u_0(x) = \frac{1}{2} x \cdot Ax$ for a non-negative definite matrix $A$ with $\text{tr}(A)=1$, corresponding to singular points.
\end{itemize}

The uniqueness of the blow-up limit type (whether it is a half-space or a polynomial) is determined by the \textit{density} of the contact set:
\[ \theta(x^0) := \lim_{r \to 0+} \frac{|B_r(x^0) \cap \{u=0\}|}{|B_r|}. \]
If $\theta(x^0) > 0$, the point is regular; if $\theta(x^0) = 0$, the point is singular.

Following the introduction of improvement of flatness and blow-up analysis, several other powerful tools have been developed. Caffarelli later combined the \textit{boundary Harnack principle} \cite{ac85} to prove $C^{1,\alpha}$ regularity of the regular set \cite{c98}. This was further refined to $C^\infty$ by De Silva and Savin \cite{dess15} using higher-order boundary Harnack inequalities. In 1999, Weiss \cite{w99} introduced the \textit{epiperimetric inequality}, a tool that has since been indispensable for treating vector-valued, degenerate, and parabolic versions of the obstacle problem \cite{w00, asuw15, fsw21, afsw22}.

The structure of the singular set is significantly more complex. For the classical obstacle problem, Sakai showed that singular points are isolated cusps in two dimensions \cite{sak91, sak93}. In higher dimensions, Schaeffer \cite{sch77} demonstrated that the free boundary can exhibit intricate singularities, such as one-sided or two-sided cusps. Despite this complexity, Caffarelli \cite{c98} proved that the singular set is contained within a union of $C^1$ manifolds of dimension $d-1$. Recent breakthroughs by Figalli and Serra \cite{fs19}, utilizing the Almgren frequency monotonicity formula, have shown that the singular set (up to a small set of higher codimension) is contained in a $C^{1,1}$ (or $C^2$ in 2D) manifold, providing a high-order asymptotic expansion for the solution near these points.

In this paper, we investigate the singular obstacle-type problem with a logarithmic term \eqref{functional}. The existence and regularity of minimizers, along with optimal growth estimates and non-degeneracy, were previously established in \cite{qs17}. Subsequent research has extended these directions to more complex variants, such as the two-phase logarithmic singular obstacle problem \cite{ks21} and highly singular versions \cite{fk24}. However, due to the inherent singularity and non-homogeneity of the logarithmic term, these studies have largely been confined to the existence and optimal regularity of minimizers.

More recently, Allen, Kriventsov and Shahgholian explored the regularity of the free boundary near regular points for a class of problems involving logarithmic terms in \cite{aks24}. In our previous work \cite{dz2}, we also addressed the regularity of the free boundary specifically for the present problem in the vicinity of regular points. Nevertheless, the geometric analysis of the singular set remained a formidable challenge that we were previously unable to overcome. 

The logarithmic forcing term creates two difficulties that are absent in the
classical obstacle problem. First, the natural scaling involves the slowly varying
factor \(1-2\log r\), so the Weiss-type energy is no longer scale invariant.
Second, the logarithmic correction terms that appear in the monotonicity formula
are not directly suited for the usual epiperimetric-decay argument. Our strategy
is to combine a modified Weiss energy functional with a logarithmic epiperimetric
inequality near the singular points, and to introduce an auxiliary correction
term \(T\) (see \eqref{eq-t}) that absorbs the non-integrable term in the decay estimate.
This allows us to recover a quantitative decay of the corrected energy and,
ultimately, the uniqueness of blow-up at singular point, and geometric structure of singular set.

\subsection{Main results and plan of this paper}\ 

Next, we will introduce the related results on the obstacle problem. In \cite{qs17}, the optimal regularity of minimizers of \eqref{functional} near the free boundary was shown to be $C^{1,\log}_{\mathrm{loc}}(\Omega)$; specifically,
\begin{equation*}
	|\nabla u(x)|\le C d(x)\log\frac{1}{d(x)},\quad x\in\Omega'\subset\subset\Omega,
\end{equation*}
where \(d(x)=\mathrm{dist}(x,\partial\{u>0\})\) and \(C>0\) is a constant. In particular, de Queiroz and Shahgholian \cite{qs17} established the following key results concerning growth estimates and non-degeneracy of minimizers for the singular problem \eqref{euler Lagrange}.
\begin{lemma}
	(Growth estimates and non-degeneracy of minimizers \cite{qs17})\label{growth and decay}
	Let \(u\) be a minimizer of \(\mathcal{L}(u; \Omega)\) in \(\mathcal{K_{\varphi}}\) and let \(x^0\in \Omega\cap \partial \{u>0\}\). Then we have
	\begin{equation*}
		C_1\le \sup_{B_r (x^0)} \frac{u(x^0+rx)}{ r^2|\log r|}\le C_2,
	\end{equation*}
	for some constants \(C_1,C_2>0\), provided \(B_r(x^0)\subset\subset\Omega\) and \(0<r\le r_0\), where \(r_0<1\) depends only on \(\Omega\), \(\varphi\), and \(d\).
\end{lemma}

These facts can be used to show that, near each free boundary point, every solution exhibits superquadratic growth of order \(r^2|\log r|\).

\begin{definition}[Blow-up limit \cite{dz2}]\label{blow-up}
	Let $u$ be a solution of \eqref{euler Lagrange} in $B_{r_0}(x^0) \subset \Omega$, with $x^0 \in \mathscr{F}(u)$. For any $x \in B_{r_0}(x^0)$, define the \textit{blow-up sequence}
	$$u_r(x) := \frac{u(x^0 + r x)}{\mu(r)},$$
	where $\mu(r)=r^2(1-2\log r)$. If $u_r(x)$ converges weakly to $u_0(x)$ in $W^{1,2}_{\rm loc}(\mathbb{R}^d)$, we say that $u_0$ is a \textit{blow-up limit of $u$ at $x^0$}.
\end{definition}

\begin{remark}
	In fact, the choice $\mu(r) = -r^2 \log r$ is also permissible; however, we adopt the current form to remain consistent with the order (or scaling) used in \cite{qs17, dz2}.
\end{remark}

Due to the lack of scaling properties in the obstacle-type problem with a logarithmic term studied here, we recall the Weiss adjusted boundary energy and Weiss-type monotonicity formula with a variable parameter $\alpha(r)$ previously introduced in \cite{qs17}.

\begin{definition} [Weiss-adjusted boundary energy]\label{weissenergy}
	Let $u$ be a solution of the obstacle problem \eqref{euler Lagrange} in $B_{r_0}(x^0)$. Then one can define the following \textit{Weiss-adjusted boundary energy} with variable parameter $\alpha(r)$
	\begin{equation}\label{weiss energy}
		\begin{aligned}
			W(r;u, x^0):=&\frac{\alpha(r)}{r^{d+2}(1-2\log r)^{2}}\mathcal{L}_0(u;B_r(x^0))\\
			&-\frac{1}{r^{d+3}(1-2\log r)^{2}}\int_{\partial B_r(x^0)}u^2 d\mathcal{H}^{d-1},
		\end{aligned}
	\end{equation}
	for $0<r\leq r_0<1$, where $$\mathcal{L}_0(u;B_r(x^0))=\int_{B_r(x^0)} \frac{1}{2}|\nabla u|^2+F(u)dx,$$
	\begin{align}\label{alp}
		\alpha(r)=1-\displaystyle\frac{1}{2\log r},
	\end{align}
	and $B_r(x^0)$ denotes an open ball with radius $r$ in $\mathbb{R}^d$ centered at $x^0$, and $\mathcal{H}^{d-1}$ denotes ${(d-1)}$-dimensional Hausdorff-measure.
\end{definition}

Indeed, applying a coordinate transformation to the Weiss-adjusted boundary energy yields
\begin{equation*}
	W(r;u,x^0)=\alpha(r) \int_{ B_1}\left(\frac{1}{2}|\nabla u_r|^2+{G}(r;u_r)\right) dx-\int_{\partial B_1} u_r^2\, d\mathcal{H}^{d-1},
\end{equation*}
where ${G}(r;v)$ is defined for any $v\geq 0$ as
\begin{equation}\label{G(r;v)}
	{G}(r;v)=\displaystyle\frac{v}{1-2\log r}\left(-\log \left(vr^2(1-2\log r)\right) +1\right).
\end{equation}
To overcome this difficulty and to meet the requirements for establishing an epiperimetric inequality, we introduce a new functional $M(r;v)$ that incorporates the additional dependence on the radial parameter $r$,
\begin{align}\label{M(r;v)}
	M(r;v):=\alpha(r)\int_{B_1}\left(\frac{1}{2}|\nabla v|^2+{G}(r;v)\right)dx-\int_{\partial B_1} v^2 d\mathcal{H}^{d-1}.
\end{align}
Moreover, we have
\begin{align}\label{M_0(v)}
	\displaystyle\lim_{r\to 0+} M(r;v) =M_0(v):=\int_{B_1}\left(\frac{1}{2}|\nabla v|^2+v\right)dx-\int_{\partial B_1} v^2 d\mathcal{H}^{d-1}.
\end{align}
We will prove that $M_0(v)$ satisfies an energy contraction property via an epiperimetric inequality, which implies the uniqueness of the blow-up limit.

\begin{lemma} {(Weiss-type monotonicity formula) }\label{monotonicity}
	Let $u$ be a solution of \eqref{euler Lagrange} in $B_{r_0}(x^0)$. Then 
	\begin{equation*}
		\frac{dW(r;u,x^0)}{dr}=\frac{\alpha (r)}{r}\int_{\partial B_1} \left(\nabla u_r\cdot x -\frac{2}{\alpha(r)}u_r\right)^2d\mathcal{H}^{d-1}+I(r;u,x^0),
	\end{equation*}
	where
	$\alpha(r)=1-\displaystyle\frac{1}{2\log r}$, and
	\begin{equation}\label{inter I}
		\begin{aligned}
			I(r;u,x^0)=&\frac{1}{2r(\log r)^{2}}\int_{B_1} \left(\frac{1}{2}|\nabla u_r|^2 +\frac{u_r}{1-2\log r}\left(-\log \left(u_r r^2(1-2\log r)\right)+1\right)\right) dx\\ &+\left(1-\frac{1}{2\log r}\right)\int_{B_1}\left(\frac{2 u_r}{r(1-2\log r)^2}\left(-\log \left(u_r(1-2\log r)\right)+1\right)\right) dx.
		\end{aligned}
	\end{equation}
	
\end{lemma}

As established in our previous work \cite{dz2}, the blow-up limit defined in Definition \ref{blow-up} satisfies the classical obstacle equation as $r \to 0+$. Consequently, the blow-up limits are subject to the aforementioned Caffarelli’s dichotomy theorem. Based on this observation, we now provide the formal definitions of half-space solutions (associated with regular points) and quadratic polynomial solutions (associated with singular points).

\begin{definition}[Half-space solutions and regular points]
	The set of half-space solutions is defined by
	\begin{equation}\label{eq3.1}
		\mathbb{H}:=\left\{ h_{\nu}=\frac{1}{2}\left(\max(x\cdot\nu,0)\right)^2 \;:\; \nu\in\mathbb{R}^d,\ |\nu|=1 \right\}.
	\end{equation}
	The set of regular points is defined by
	\begin{align*}
		\mathcal{R}_u :=\left\{x^0\in \mathscr{F} (u):  \lim_{r \to 0+} \frac{u(x^0 + r x)}{r^2(1-2\log r)}\in\mathbb{H}\right\},
	\end{align*}
	where the limit is taken over all blow-up sequences.
\end{definition}

\begin{definition}[Quadratic polynomial solutions and singular points]
	The set of Quadratic polynomial solutions is defined by
	\begin{equation}\label{eq3.1}
		\mathbb{K}:=\left\{ Q_A(x)=\frac{1}{2}x \cdot A x:  \text{$A$ is a symmetric non negative matrix
			with }  \ \operatorname{tr}(A)=1 \right\}.
	\end{equation}
	The set of singular points is defined by
	\begin{align*}
		\mathcal{S}_u :=\left\{x^0\in \mathscr{F} (u):  \lim_{r_j \to 0+} \frac{u(x^0 + r_j x)}{r_j^2(1-2\log r_j)}\in\mathbb{K}\right\},
	\end{align*}
	where the limit is taken for some blow-up sequence.
\end{definition}	

\begin{remark} Let $\omega_d=\mathcal{H}^{d-1}(\partial B_1)$, a direct calculation yields the following values:
	\begin{itemize}
		\item for any $Q_A\in \mathbb{K}$, $M_0(Q_A)=\frac{\omega_d}{4d(d+2)}:=\Theta$;
		\item for any $q_{\nu}\in \mathbb{H}$, $M_0(q_{\nu})=\frac{\Theta}{2}$.
	\end{itemize}
\end{remark}

Our previous research has established the $C^{1,\beta}$ regularity ($\beta \in (0,1)$) of the free boundary near regular points for the singular obstacle-type problem with a logarithmic term \eqref{functional}. The present paper aims to investigate the geometric structure of the singular set. As pointed out in the pioneering work of Weiss \cite{w99}, the geometric analysis of the singular set in higher dimensions can be achieved through an epiperimetric inequality. Specifically, a log-epiperimetric inequality (i.e. logarithmic epiperimetric inequality) is required, which typically necessitates a proof via the direct method. This stands in contrast to the standard epiperimetric inequality used in our previous regularity analysis, which was established by contradiction. To this end, we first provide the construction of the logarithmic epiperimetric inequality and subsequently prove it using the direct method. As noted in \cite{csv18}, such inequalities can be obtained by constructing appropriate competitors. However, since the functional \eqref{functional} involves a logarithmic term and is significantly more complex than the classical obstacle problem, we must overcome substantial technical hurdles in the direct proof. Furthermore, since the Weiss energy in our context requires a correction term (the subtraction of an integral of an integrable function) to restore monotonicity, the corresponding log-epiperimetric inequality must be adjusted accordingly. Finally, due to the loss of scaling invariance caused by the logarithmic term, ``bad'' (i.e., non-integrable) terms inevitably arise in the subsequent energy decay estimates. To resolve this issue, we introduce an additional modification to the log-epiperimetric inequality—specifically, the auxiliary function $T$ defined below. By incorporating this modified inequality into the decay estimates, we can successfully handle the non-integrable terms and complete the analysis.

To establish the geometric structure of the singular set, the key analytical
input is the following {\it modified log-epiperimetric inequality} near singular
points.

\begin{theorem}[Modified log-epiperimetric inequality]\label{thm19}
	There are $\delta(d) > 0$ and $\epsilon > 0$ such that the following claim holds. For every non-negative function $c\in W^{1,2}(\partial B_1)$, with $2$-homogeneous extension $z$ on $B_1$, satisfying 
	\[
	\operatorname{dist}_{L^2(\partial B_1)}(c, \mathbb{K}) \leq \delta \quad \text{and} \quad M_0(z) - \Theta \leq 1;
	\]
	and for every minimizer $u$ of \eqref{functional}.
	Then there exists a non-negative function $v\in H^1(B_1)$ and $v=c$ on $\partial B_1$, for any $s\in(0,r_0]$ ($0<r_0\le 1$), satisfying
	\begin{align}\label{eq-log-epi}
			M_I(s; v)- \Theta  
			\leq \left(M_I(s; z)- \Theta\right) \left(1 - \epsilon\left|  M_I(s; z)- \Theta\right|^\gamma \right)+ T(s; z),
	\end{align}
	where 
\[
\gamma=
\begin{cases}
		0, & d=2,\\[2mm]
		\dfrac{d-1}{d+3}, & d\ge 3,
\end{cases}
\]
	\begin{align}\label{eq-t}
	T(s; z) := T(s;c)=\frac{1}{(d+2)^2} \int_{\partial B_1} \frac{c}{\log s} \, d\mathcal{H}^{d-1},
	\end{align}
	\[M_I(r;v)=M(r;v)-\int_{0}^{r}I(\rho;u,x_0)d\rho.\]
\end{theorem}

\begin{remark}
	Theorem \ref{thm19} should be viewed as the singular point analogue, in the present
	logarithmic obstacle setting, of the log-epiperimetric inequality of
	Colombo--Spolaor--Velichkov \cite{csv18}, and at the same time as a direct method
	counterpart of the homogeneity improvement philosophy initiated by Weiss \cite{w99}.
	Compared with the classical obstacle problem, the main new difficulty is that
	the forcing term $-\log u$ destroys exact scaling, so the Weiss functional has
	to be replaced by the corrected quantity $M_I$, and the proof must keep track
	simultaneously of the variable-parameter correction, the remainder
	$I(r;u,x_0)$, and the auxiliary lower-order term $T(s;z)$. In particular, the classical linear contraction $E\mapsto(1-\varepsilon)E$ is replaced by the logarithmic contraction 
	\[
	E \mapsto (1-\varepsilon |E|^\gamma)E,
	\]
	which is precisely what leads to logarithmic decay of the excess and hence to
	a $C^{1,\log}$-type description of the singular strata. In dimension $d=2$,
	where $\gamma=0$, one recovers the classical epiperimetric regime and therefore
	a H\"older-type decay, in agreement with the two dimensional behaviour already
	visible in \cite{csv18}.
\end{remark}

We can now state the main result of this paper on the geometric structure of
the singular free boundary.

Before stating it, for each $k=0,\dots,d-1$ we define the $k$-th singular
stratum by
\[
S_{k,u}:=
\Bigl\{
x\in S_u:\ \dim(\ker A)\le k
\text{ for every blow-up } Q_A\in K \text{ of }u\text{ at }x
\Bigr\}.
\]
Once uniqueness of blow-up is available, this is equivalently written as
\[
S_{k,u}
=
\bigcup_{l=0}^{k}
\Bigl\{
x\in S_u:\ \dim(\ker A)=l
\text{ for the unique blow-up } Q_A\in K \text{ of }u\text{ at }x
\Bigr\}.
\]
In particular, for the lowest stratum $S_{0,u}$, the logarithmic modulus
improves to a H\"older modulus in dimension two.

\begin{theorem}[Geometric structure of the singular set of the free boundary]\label{regularity of FB}
	Let $\Omega\subset\mathbb{R}^d$ be an open set and let $u\in H^1(\Omega)$ be a
	minimizer of \eqref{functional}. Then, for every $k=0,\dots,d-1$ and every
	$x_0\in S_{k,u}$, there exists $r_0=r_0(x_0)>0$ such that
	$S_{k,u}\cap B_{r_0}(x_0)$ is contained in a single $k$-dimensional
	submanifold of class $C^{1,\log}$.
	
	Moreover, for every open set $\Omega_0\Subset\Omega$, there exists
	$C=C(\Omega_0)>0$ such that
\begin{equation}\label{eq-regular of Fb}
\|Q_{A(x_1)}-Q_{A(x_2)}\|_{L^2(\partial B_1)}
\le
C(-\log|x_1-x_2|)^{-\frac{1-\gamma}{2\gamma}}
\qquad\text{for every }x_1,x_2\in S_u\cap\Omega_0,
\end{equation}
	where $A(x)$ is the matrix corresponding to \(Q_{A(x)}\), and \(Q_{A(x)}\) is the unique blow-up of \(u\) at \(x\).
	
	If $d=2$, then
	\[
	\|Q_{A(x_1)}-Q_{A(x_2)}\|_{L^2(\partial B_1)}
	\le
	C|x_1-x_2|^\beta
	\qquad\text{for every }x_1,x_2\in S_{k,u}\cap\Omega_0,\ k=0,1,
	\]
	for some $\beta\in(0,1)$ (as in \cite{dz2}). In particular, $S_{0,u}$ consists of isolated points,
	and $S_{1,u}$ is locally contained in a one-dimensional $C^{1,\beta}$
	submanifold.
\end{theorem}

\begin{remark}
	Theorem \ref{regularity of FB} is the natural counterpart, for the logarithmic obstacle problem,
	of the singular-strata regularity theorem of Colombo--Spolaor--Velichkov \cite{csv18}
	for the classical obstacle problem. At the level of geometric conclusion, the
	higher dimensional statement is of the same type---namely, a $C^{1,\log}$
	description of the singular strata---but the mechanism is different: here the
	decay comes from the corrected Weiss excess $W_I$ together with the auxiliary
	term $T$, both of which are needed to compensate for the loss of exact scaling
	caused by the term $-\log u$.
\end{remark}

\begin{remark}
	In dimension two, Theorem \ref{regularity of FB} yields only $C^{1,\beta}$ regularity of the
	singular set, whereas Figalli--Serra \cite{fs19} proved for the classical obstacle
	problem that singular points are locally contained in a $C^2$ curve. From the
	viewpoint of the present argument, this gap is mainly methodological: the
	log-epiperimetric inequality gives quantitative control of the quadratic
	blow-up and of its continuity, but it does not provide the higher-order
	asymptotic expansion that is needed for a $C^2$ description. At the same time,
	the logarithmic forcing creates genuine additional technical difficulties: most
	notably the loss of exact scaling and the appearance of correction terms in the
	monotonicity formula, so extending the strategy of \cite{fs19} to the present setting
	would require tools beyond the current epiperimetric framework.
\end{remark}

The rest of this paper is organized as follows. In Section 2, we prove
Theorem \ref{thm19}, namely the modified log-epiperimetric inequality near singular
points. In Section 3, we derive the logarithmic decay of the corrected Weiss
excess and prove uniqueness of blow-up at singular points. Finally, in
Section 4, we prove Theorem \ref{regularity of FB} and establish the geometric structure of the
singular strata of the free boundary.

\section{Proof of the log-epiperimetric inequality}

In this section we prove Theorem \ref{thm19}. The main idea is that: let $z=r^2c$, then decompose $c$ in Fourier series on the $\partial B_1$. The argument follows the four-step
strategy of \cite{csv18}, adapted to the present logarithmic setting. We first
derive the basic energy identity for the decomposition
\(z=q_\nu+Q_A+\phi\), then choose a positive quadratic part \(Q_B\), next improve
the homogeneity of the error term, and finally control the remaining second
modes by the higher modes. The only genuinely new feature with respect to the
classical obstacle problem is the presence of the logarithmic correction terms,
which are handled through the modified energy \(M_I\) and the auxiliary term \(T\).

Recalling that eigenvalues and eigenfunctions on subdomains of the sphere. Let \( S \subseteq \mathbb{S}^{d-1} \) be an open set.  
Let \( 0 < \lambda_1 \leq \lambda_2 \leq \cdots \leq \lambda_j \leq \cdots \) be the eigenvalues (counted with multiplicity) of the spherical Laplace–Beltrami operator with Dirichlet conditions on \( \partial S \) and \( \{ \phi_j \}_{j \geq 1} \) be the corresponding eigenfunctions, that are the solutions of the problem
\[
-\Delta_{\mathbb{S}^{d-1}} \phi_j = \lambda_j \phi_j \quad \text{in } S, \qquad \phi_j = 0 \quad \text{on } \partial S,
\]
\begin{equation}\label{eq24}
\int_S \phi_j^2(\theta) \, d\mathcal{H}^{d-1}(\theta) = 1. 
\end{equation}
Any function \( \psi \in H_0^1(S) \) can be decomposed as \( \psi(\theta) = \sum_{j=1}^\infty c_j \phi_j(\theta) \). 

We can assume without loss of generality that \( W(z) - \Theta \geq 0 \), since otherwise the statement is true with \( h = z \). Given any two-homogeneous function \( z(r, \theta) = r^2 c(\theta) \), we can decompose it in Fourier series on the sphere \( \partial B_1 \) as
\[
c(\theta) = \sum_{j=1}^\infty c_j \phi_j(\theta) 
= c_1 \phi_1 + \sum_{\{j : \lambda_j = d-1\}} c_j \phi_j(\theta)
+ \sum_{\{j : \lambda_j = 2d\}} c_j \phi_j(\theta) + \sum_{\{j : \lambda_j > 2d\}} c_j \phi_j(\theta).
\]
Therefore \( z \) can be decomposed in a unique way as
\[
z = q_\nu + Q_A + \varphi,
\]
where
\begin{itemize}
	\item[(i)] \( \nu \in \mathbb{R}^d \) is such that \( q_\nu(x) =\frac{1}{2} (x \cdot \nu)_+^2 \) contains in its Fourier expansion precisely the sum
	\[
	\sum_{\{j : \lambda_j = d-1\}} c_j \phi_j(\theta);
	\]
	\item[(ii)] \( A \) is a symmetric matrix depending on the coefficients \( c_j \), corresponding to the eigenvalues \( \lambda_j = 0, 2d \), and \( Q_A(x) = \frac{1}{2}x \cdot Ax \);
	\item[(iii)] \( \varphi \) is a two-homogeneous function, in polar coordinates \( \varphi(r, \theta) = r^2 \phi(\theta) \), containing only higher modes on \( \partial B_1 \), that is the trace \( \phi \) can be written in the form
	\[
	\phi(\theta) = \sum_{\{j : \lambda_j > 2d\}} c_j \phi_j(\theta),
	\]
	where \( \{\phi_j\}_{j \in \mathbb{N}} \) are the eigenfunctions of the spherical Laplacian as in \eqref{eq24} with \( S = \partial B_1 \).
\end{itemize}

Notice that, 
\begin{align*}
	z=q_{\nu}+Q_A+\varphi.
\end{align*}
Since we do not know whether the above representation $A$ is positive definite, let 
$B$ be a symmetric non-negative definite matrix and define $Q_B(x)=\frac{1}{2}x\cdot Bx$. Then $z$ can be expressed as,
\begin{align*}
	z=q_{\nu}+Q_B+\psi,
\end{align*}
where the 2-homogeneous function $\psi:=Q_A-Q_B+\varphi$, and denote $\alpha$-homogeneous function $\tilde{\psi}$ with the same boundary values as $\psi$.

Therefore, we next study the decomposition of the energy under the decomposition of the function $z$.

\subsection{Decomposition of the energy}\

We denoted that
\begin{align*}
	\tilde{M}(r;v)=\alpha(r)\int_{ B_1}\frac{1}{2}|\nabla v|^2dx-\int_{ \partial B_1} v^2 d\mathcal{H}^{d-1}.
\end{align*}
First, based on the decomposition of the functions $z$, we set $v=q_{\nu}+Q_B+\tilde{\psi}$, and then obtain the corresponding energy decomposition. However, due to the lack of scaling invariance caused by the logarithmic term in the functional, on one hand the modified Weiss energy functional involves a variable parameter $\alpha(r)$, and on the other hand it gives rise to a complicated function $G(r,v)$ as in \eqref{G(r;v)}. This makes the problem much more complex than the classical obstacle problem, and consequently our proof becomes difficult to carry out. Fortunately, from the limit of the functional \eqref{M_0(v)} we know that when $r$ is sufficiently small, the modified Weiss energy functional coincides with the Weiss energy functional of the classical obstacle problem. Therefore, we aim to apply the modified log-epiperimetric inequality in this regime for sufficiently small $s$. As a first step, we directly compute the energy decomposition in this complicated setting, so as to facilitate further investigation.

\begin{lemma}
	Let \(\alpha > 2\) and \(\varepsilon_\alpha = \frac{\alpha-2}{d+\alpha}\); let \(0 \neq \nu = (\nu_1, \ldots, \nu_d) \in \mathbb{R}^d\), \(q_\nu(x) =\frac{1}{2} (x \cdot \nu)_+^2\), \(c_0 = \sum_{j=1}^d \nu_j^2\); let \(B\) be a symmetric matrix with \(b = \operatorname{tr} B \neq 0\) and
	\[
	Q_B(x) = \frac{1}{2}x \cdot Bx.
	\]
	Suppose that \(\varphi \in H^1(\partial B_1)\), \(\varphi(r, \theta) = r^2 \phi(\theta)\) and \(\tilde{\varphi}(r, \theta) = r^\alpha \phi(\theta)\). Then
	\begin{align*}
		&M_I(s;q_\nu + Q_B + \tilde{\varphi}) - \Theta -T(s;q_\nu + Q_B + \varphi)- (1 - \varepsilon_\alpha) \left( M_I(s; q_\nu + Q_B + \varphi) - \Theta \right)\\ =\ &-\alpha(s)\epsilon_{\alpha}\Theta\left(\frac{(1-b-c_0)^2+(1-b)^2}{2}\right)+\epsilon_{\alpha}\Bigg(-\frac{1}{2\log s }\Theta-\frac{b^2}{2\log s}\int_{ \partial B_1}Q^2 \ dx \\
		&-\int_{ B_1}\frac{\left(bQ+c_0q+\tilde{\varphi}\right)\log\left(\left(bQ+c_0q+\tilde{\varphi}\right)(1-2\log s)\right)}{1-2\log s} dx\\
		&+(1-\epsilon_{\alpha})\int_{ B_1}\frac{\left(bQ+c_0q+\varphi\right)\log\left(\left(bQ+c_0q+\varphi\right)(1-2\log s)\right)}{1-2\log s} dx\nonumber\\
		&-\frac{c_0^2}{2\log s}\int_{ \partial B_1} q^2 \ d\mathcal{H}^{d-1}\Bigg)+\epsilon_{\alpha}\int_{0}^{s}I(\rho;u,x^0)\ d\rho-\frac{2b\epsilon_{\alpha}}{\log s}\int_{ \partial B_1}c_0Qq d\mathcal{H}^{d-1}\\
		&-T(s;bQ+c_0q+\psi)+\tilde{M}(s;\tilde{\psi})-(1-\epsilon_{\alpha})\tilde{M}(s;\psi).
	\end{align*}

\end{lemma}

\begin{proof}
	Let $Q=\frac{1}{b}Q_B$ ( such that $Q\in\mathbb{K}$), then for every $Q\in\mathbb{K}$,
	\[M_0(Q)=\frac{\omega_d}{4d(d+2)}=:\Theta,\]
	and that for non-negative $bQ+\eta$,
	\begin{align*}
		&M(s;bQ+\eta)-\Theta-\int_{0}^{s}I(\rho;u,x^0)d\rho\\
		=\ &\alpha(s)\int_{ B_1}\frac{1}{2}|\nabla(bQ+\eta)|^2+\frac{bQ+\eta}{1-2\log s}\left(1-\log\left((bQ+\eta)\mu(s)\right)\right) \ dx\\
		&-\int_{ \partial B_1}(bQ+\eta)^2\ d\mathcal{H}^{d-1}-\Theta-\int_{0}^{s}I(\rho;u,x^0)d\rho\\
		=\ &\alpha(s)\int_{ B_1}\frac{1}{2}\left(|\nabla(bQ)|^2+|\nabla \eta|^2+2\nabla(bQ)\nabla\eta\right)+\frac{bQ+\eta}{1-2\log s}\left(1-\log\left((bQ+\eta)\mu(s)\right)\right)\ dx\\
		&-\int_{ \partial B_1}(bQ)^2+\eta^2+2bQ\eta\  d\mathcal{H}^{d-1}-\Theta-\int_{0}^{s}I(\rho;u,x^0)d\rho\\
		=\ &b^2\tilde{M}(s;Q)+\tilde{M}(s;\eta)+b\left(\alpha(s)\int_{B_1}\nabla Q\cdot\nabla\eta dx -\int_{ \partial B_1}2 Q \eta d\mathcal{H}^{d-1}\right)\\
		&+\alpha(s)\int_{ B_1}\frac{bQ+\eta}{1-2\log s}\left(1-\log\left((bQ+\eta)\mu(s)\right)\right)\ dx-\Theta-\int_{0}^{s}I(\rho;u,x^0)d\rho.
	\end{align*}
	Using integration by parts, we have \begin{align*}
		\int_{ B_1}\nabla Q\cdot\nabla\eta dx =\int_{ \partial B_1}2 Q \eta d\mathcal{H}^{d-1}-\int_{ B_1}\eta dx.
	\end{align*}
	This implies that
	\begin{align}\label{eq-decom 1}
		&M(s;bQ+\eta)-\Theta-\int_{0}^{s}I(\rho;u,x^0)d\rho\\
		=\ &b^2\tilde{M}(s;Q)+\tilde{M}(s;\eta)-\alpha(s)b\int_{ B_1}\eta dx-\frac{2b}{\log s}\int_{ \partial B_1}Q\eta d\mathcal{H}^{d-1} \nonumber\\	&+\alpha(s)\int_{ B_1}\frac{bQ+\eta}{1-2\log s}\left(1-\log\left((bQ+\eta)\mu(s)\right)\right)\ dx-\Theta\nonumber\\
		&-\int_{0}^{s}I(\rho;u,x^0)d\rho.\nonumber
	\end{align}
	Note that \begin{align*}
		\Theta:=\ &M_0(Q)=\int_{ B_1}\frac{1}{2}|\nabla Q|^2+Q dx -\int_{ \partial B_1} Q^2\  d\mathcal{H}^{d-1}\\
		=\ &\tilde{M}_0(Q)+\int_{ B_1} Q\ dx\\
		=\ &\frac{1}{2}\int_{ B_1}Q\ dx,
	\end{align*}
	where the equality follows from integration by parts and the homogeneity of $Q$ of degree 2,
	\begin{align*}
		\tilde{M}_0(Q):=\int_{ B_1}\frac{1}{2}|\nabla Q|^2dx -\int_{ \partial B_1} Q^2 \ d\mathcal{H}^{d-1}=- \frac{1}{2}\int_{ B_1} Q \ dx.
	\end{align*}
	Furthermore, it can be seen that
	\begin{align*}
		\tilde{M}(s;Q)=\ &\alpha(s)\int_{ B_1}\frac{1}{2}|\nabla Q|^2-\int_{ \partial B_1} Q^2 \ d\mathcal{H}^{d-1}\\
		=\ &\alpha(s)\tilde{M}_0(Q)-\frac{1}{2\log s}\int_{ \partial B_1} Q^2\ d\mathcal{H}^{d-1}\\
		=\ &-\alpha(s)\Theta-\frac{1}{2\log s}\int_{ \partial B_1} Q^2 d\mathcal{H}^{d-1}.
	\end{align*}
	Substituting \eqref{eq-decom 1} into the above, we obtain
	\begin{align}\label{eq-decom 2}
		&M(s;bQ+\eta)-\Theta-\int_{0}^{s}I(\rho;u,x^0)d\rho\\
		=\ &-b^2\alpha(s)\Theta-\frac{b^2}{2\log s}\int_{ \partial B_1}Q^2 \ d\mathcal{H}^{d-1}+\alpha(s)\int_{ B_1}\frac{bQ+\eta}{1-2\log s}\left(1-\log\left((bQ+\eta)\mu(s)\right)\right)dx\nonumber\\
		&-\Theta+\tilde{M}(s;\eta)-\alpha(s)b\int_{ B_1}\eta dx-\frac{2b}{\log s}\int_{ \partial B_1}Q\eta d\mathcal{H}^{d-1}  -\int_{0}^{s}I(\rho;u,x^0)d\rho.\nonumber
	\end{align}
	Note that the aforementioned equality contains a logarithmic term, which lacks scaling invariance. To address this, we judiciously employ higher-order infinitesimal terms to handle this term in a small neighborhood of the free boundary:
		\begin{align}\label{eq-decom 3}
		&\alpha(s)\int_{ B_1}\frac{bQ+\eta}{1-2\log s}\left(1-\log\left((bQ+\eta)\mu(s)\right)\right)\ dx\\
		=\ &\alpha(s)\int_{ B_1}bQ+\eta -\frac{\left(bQ+\eta\right)\log\left(\left(bQ+\eta\right)(1-2\log s)\right)}{1-2\log s} dx\nonumber\\
		=\ &b\alpha(s)\int_{ B_1}Q\ dx+\alpha(s)\int_{ B_1}\eta \ dx-\alpha(s)\int_{ B_1}\frac{\left(bQ+\eta\right)\log\left(\left(bQ+\eta\right)(1-2\log s)\right)}{1-2\log s}dx \nonumber\\
		=\ &2b\alpha(s)\Theta+\alpha(s)\int_{ B_1}\eta\ dx-\alpha(s)\int_{ B_1}\frac{\left(bQ+\eta\right)\log\left(\left(bQ+\eta\right)(1-2\log s)\right)}{1-2\log s}dx.\nonumber
	\end{align}
	Thus, \eqref{eq-decom 2} can be expressed as
	\begin{align*}
		&M(s;bQ+\eta)-\Theta-\int_{0}^{s}I(\rho;u,x^0)d\rho\nonumber\\
		=\ &-\Theta\left(b^2-2b+1\right)\alpha(s)+\left(\alpha(s)-1\right)\Theta-\frac{b^2}{2\log s}\int_{\partial B_1} Q^2\ d\mathcal{H}^{d-1}\\ 
		&-\alpha(s)\int_{ B_1}\frac{\left(bQ+\eta\right)\log\left(\left(bQ+\eta\right)(1-2\log s)\right)}{1-2\log s} dx\\
		&+\tilde{M}(s;\eta)-\alpha(s)b\int_{ B_1}\eta dx -\frac{2b}{\log s}\int_{ \partial B_1}Q\eta d\mathcal{H}^{d-1}  -\int_{0}^{s}I(\rho;u,x^0)d\rho.
	\end{align*}
	
	Then, set $\eta=c_0q+\psi$, $q=\frac{1}{c_0}q_{\nu}$, since $c_0=\sum_{j=1}^{d} \nu_j^2$, it  follows that $q\in\mathbb{H}$. It further appears that
	\begin{align}\label{M-theta}
		&M(s;bQ+c_0q+\psi)-\Theta-\int_{0}^{s}I(\rho;u,x^0)\ d\rho\\
		=\ &-(1-b)^2\alpha(s)\Theta-\frac{1}{2\log s}\Theta -\frac{b^2}{2\log s}\int_{ \partial B_1}Q^2 \ d\mathcal{H}^{d-1} \nonumber\\
		&-\alpha(s)\int_{ B_1}\frac{\left(bQ+c_0q+\psi\right)\log\left(\left(bQ+c_0q+\psi\right)(1-2\log s)\right)}{1-2\log s} dx\nonumber\\
		&+\tilde{M}(s;c_0q+\psi)+\alpha(s)(1-b)\int_{ B_1}c_0q+\psi \ dx-\frac{2b}{\log s}\int_{ \partial B_1}Q(c_0q+\psi) d\mathcal{H}^{d-1}\nonumber\\ &-\int_{0}^{s}I(\rho;u,x^0)d\rho\nonumber\\
		=\ &-(1-b)^2\alpha(s)\Theta-\frac{1}{2\log s}\Theta -\frac{b^2}{2\log s}\int_{ \partial B_1}Q^2 \ d\mathcal{H}^{d-1}\nonumber\\ &-\alpha(s)\int_{ B_1}\frac{\left(bQ+c_0q+\psi\right)\log\left(\left(bQ+c_0q+\psi\right)(1-2\log s)\right)}{1-2\log s} dx\nonumber\\
		&+\alpha(s)\int_{ B_1}\frac{1}{2}|\nabla (c_0 q+\psi)|^2 \ dx-\int_{ \partial B_1}(c_0q+\psi)^2\ d\mathcal{H}^{d-1}\nonumber\\
		&+\alpha(s)(1-b)\int_{ B_1}c_0q+\psi\ dx-\frac{2b}{\log s}\int_{ \partial B_1}Q(c_0q+\psi) d\mathcal{H}^{d-1}-\int_{0}^{s}I(\rho;u,x^0)d\rho\nonumber\\
		=\ &-(1-b)^2\alpha(s)\Theta-\frac{1}{2\log s}\Theta -\frac{b^2}{2\log s}\int_{ \partial B_1}Q^2 \ d\mathcal{H}^{d-1}\nonumber\\ &-\alpha(s)\int_{ B_1}\frac{\left(bQ+c_0q+\psi\right)\log\left(\left(bQ+c_0q+\psi\right)(1-2\log s)\right)}{1-2\log s} dx\nonumber\\
		&+c_0^2\tilde{M}(s;q)+\tilde{M}(s;\psi)+\alpha(s)\int_{ B_1}c_0\nabla q\cdot\nabla\psi\ dx-2c_0\int_{ \partial B_1} q\psi d\mathcal{H}^{d-1}\nonumber\\
		&+\alpha(s)(1-b)\int_{ B_1}c_0q+\psi\ dx-\frac{2b}{\log s}\int_{ \partial B_1}Q(c_0q+\psi) d\mathcal{H}^{d-1}-\int_{0}^{s}I(\rho;u,x^0)d\rho\nonumber\\
		=\ &-\alpha(s)\Theta\left((1-b)^2+\frac{1}{2}c_0^2-(1-b)c_0\right)-\frac{1}{2\log s}\Theta -\frac{b^2}{2\log s}\int_{ \partial B_1} Q^2\ d\mathcal{H}^{d-1}\nonumber\\
		&-\alpha(s)\int_{ B_1}\frac{\left(bQ+c_0q+\psi\right)\log\left(\left(bQ+c_0q+\psi\right)(1-2\log s)\right)}{1-2\log s} dx\nonumber\\
		&-\frac{c_0^2}{2\log s}\int_{ \partial B_1} q^2 \ d\mathcal{H}^{d-1}+\tilde{M}(s;\psi)-\alpha(s)\int_{ B_1}c_0\psi\ dx-\frac{1}{2\log s}\int_{ \partial B_1} 2c_0q\psi d \mathcal{H}^{d-1} \nonumber\\
		&+\alpha(s)(1-b)\int_{ B_1}\psi \ dx-\frac{2b}{\log s}\int_{ \partial B_1}Q(c_0q+\psi) d\mathcal{H}^{d-1}-\int_{0}^{s}I(\rho;u,x^0)d\rho.\nonumber
	\end{align}
	We now estimate that
	\begin{align*}
		&M_I(s; bQ+c_0q+\tilde{\psi}
		)-\Theta-(1-\epsilon_{\alpha})(M_I(s;bQ+c_0 q +\psi)-\Theta) -T(s;bQ+c_0 q +\psi)\\
		=\ &-\alpha(s)\epsilon_{\alpha}\Theta\left(\frac{(1-b-c_0)^2+(1-b)^2}{2}\right)+\epsilon_{\alpha}\Bigg(-\frac{1}{2\log s }\Theta-\frac{b^2}{2\log s}\int_{ \partial B_1}Q^2 \ dx \\
		&-\alpha(s)\int_{ B_1}\frac{\left(bQ+c_0q+\tilde{\psi}\right)\log\left(\left(bQ+c_0q+\tilde{\psi}\right)(1-2\log s)\right)}{1-2\log s} dx\\
		&+(1-\epsilon_{\alpha})\alpha(s)\int_{ B_1}\frac{\left(bQ+c_0q+\psi\right)\log\left(\left(bQ+c_0q+\psi\right)(1-2\log s)\right)}{1-2\log s} dx\nonumber\\
		&-\frac{c_0^2}{2\log s}\int_{ \partial B_1} q^2 \ d\mathcal{H}^{d-1}\Bigg)+\epsilon_{\alpha}\int_{0}^{s}I(\rho;u,x^0)\ d\rho-T(s;bQ+c_0q+\psi)\\
		&+\tilde{M}(s;\tilde{\psi})-\alpha(s)\int_{ B_1}c_0\tilde{\psi}\ dx-\frac{1}{2\log s}\int_{ \partial B_1} 2c_0 q \tilde{\psi} \ d\mathcal{H}^{d-1}\\
		&+\alpha(s)(1-b)\int_{ B_1}\tilde{\psi}\ dx-\frac{2b}{\log s}\int_{ \partial B_1}Q(c_0q+\tilde{\psi}) d\mathcal{H}^{d-1}-(1-\epsilon_{\alpha})\Bigg(\tilde{M}(s;\psi)\\
		&-\alpha(s)\int_{ B_1} c_0\psi\ dx-\frac{1}{2\log  s}\int_{ \partial B_1}2 c_0 q\psi \ d\mathcal{H}^{d-1}+\alpha(s)(1-b)\int_{ B_1} \psi\ dx\\
		&-\frac{2b}{\log s}\int_{ \partial B_1}Q(c_0q+\psi) d\mathcal{H}^{d-1}\Bigg).
	\end{align*}
	In view of the choice of $\epsilon_{\alpha}$, we have
	\[1-\epsilon_{\alpha}=1-\frac{\alpha-2}{d+\alpha}=\frac{d+2}{d+\alpha},\]
	such that
	\begin{align*}
		\int_{ B_1}\tilde{\psi}\ dx=(1-\epsilon_{\alpha})\int_{ B_1}\psi\ dx.
	\end{align*}
	Therefore,
	\begin{align*}
		&M_I(s; bQ+c_0q+\tilde{\psi})-\Theta-(1-\epsilon_{\alpha})(M_I(s;bQ+c_0 q +\psi)-\Theta) -T(s;bQ+c_0 q +\psi)\\
		=\ &-\alpha(s)\epsilon_{\alpha}\Theta\left(\frac{(1-b-c_0)^2+(1-b)^2}{2}\right)+\epsilon_{\alpha}\Bigg(-\frac{1}{2\log s }\Theta-\frac{b^2}{2\log s}\int_{ \partial B_1}Q^2 \ dx \\
		&-\alpha(s)\int_{ B_1}\frac{\left(bQ+c_0q+\tilde{\psi}\right)\log\left(\left(bQ+c_0q+\tilde{\psi}\right)(1-2\log s)\right)}{1-2\log s} dx\\
		&+(1-\epsilon_{\alpha})\alpha(s)\int_{ B_1}\frac{\left(bQ+c_0q+\psi\right)\log\left(\left(bQ+c_0q+\psi\right)(1-2\log s)\right)}{1-2\log s} dx\nonumber\\
		&-\frac{c_0^2}{2\log s}\int_{ \partial B_1} q^2 \ d\mathcal{H}^{d-1}\Bigg)+\epsilon_{\alpha}\int_{0}^{s}I(\rho;u,x^0)\ d\rho-\frac{2b\epsilon_{\alpha}}{\log s}\int_{ \partial B_1}c_0Qq d\mathcal{H}^{d-1}\\
		&-T(s;bQ+c_0q+\psi)+\tilde{M}(s;\tilde{\psi})-(1-\epsilon_{\alpha})\tilde{M}(s;\psi).
	\end{align*}
\end{proof}

We decompose the above equality into three parts. \textbf{First}, for the first part: it is clear that
\begin{align*}
	-\alpha(s)\epsilon_{\alpha}\Theta\left(\frac{(1-b-c_0)^2+(1-b)^2}{2}\right)\leq 0.
\end{align*}
\textbf{Second}, the second part:
\begin{align*}
	&\epsilon_{\alpha}\Bigg(-\frac{1}{2\log s }\Theta-\frac{b^2}{2\log s}\int_{ \partial B_1}Q^2 \ dx -\frac{c_0^2}{2\log s}\int_{ \partial B_1} q^2 \ d\mathcal{H}^{d-1}\\
	&-\alpha(s)\int_{ B_1}\frac{\left(bQ+c_0q+\tilde{\psi}\right)\log \left(\left(bQ+c_0q+\tilde{\psi}\right)(1-2\log s)\right)}{1-2\log s}dx\\
	&+(1-\epsilon_{\alpha})\alpha(s)\int_{ B_1}\frac{\left(bQ+c_0q+\psi\right)\log \left(\left(bQ+c_0q+\psi\right)(1-2\log s)\right)}{1-2\log s}dx\Bigg)\\
	&+\epsilon_{\alpha}\int_{0}^{s}I(\rho;u,x^0)\ d\rho-\frac{2b\epsilon_{\alpha}}{\log s}\int_{ \partial B_1}c_0Qq d\mathcal{H}^{d-1}-T(s;bQ+c_0q+\psi).
\end{align*}

 This part also represents the most significant difference from the classical obstacle problem. The differences are mainly reflected in the lack of scaling invariance, the integrable terms arising from the logarithmic term in the Weiss monotonicity formula, and the non‑integrable term \(T\) absorbed in the proof of energy decay via the modified log-epiperimetric inequality. Together, these factors give rise to several complicated terms. The main idea for handling these terms is to first simplify them, then analyze the leading term and its sign, with the ultimate goal of ensuring that these “bad” terms are non‑negative.

Therefore, we mainly study the order of this term in $s$,
\begin{align*}
	&-\alpha(s)\int_{ B_1}\frac{\left(bQ+c_0q+\tilde{\psi}\right)\log \left(\left(bQ+c_0q+\tilde{\psi}\right)(1-2\log s)\right)}{1-2\log s}dx\\
	&+(1-\epsilon_{\alpha})\alpha(s)\int_{ B_1}\frac{\left(bQ+c_0q+\psi\right)\log \left(\left(bQ+c_0q+\psi\right)(1-2\log s)\right)}{1-2\log s}dx\\
	=&-\alpha(s)\int_{ B_1}\frac{\left(bQ+c_0q+\tilde{\psi}\right)\log \left(bQ+c_0q+\tilde{\psi}\right)+\left(bQ+c_0q+\tilde{\psi}\right)\log \left(1-2\log s\right)}{1-2\log s}dx\\
	&+(1-\epsilon_{\alpha})\alpha(s)\int_{ B_1}\frac{\left(bQ+c_0q+\psi\right)\log \left(bQ+c_0q+\psi\right)+\left(bQ+c_0q+\psi\right)\log \left(1-2\log s\right)}{1-2\log s}dx\\
	\leq &-\alpha(s)\int_{ B_1}\frac{\left(bQ+c_0q+\tilde{\psi}\right)\log \left(bQ+c_0q+\tilde{\psi}\right)+\left(bQ+c_0q+\tilde{\psi}\right)\log \left(1-2\log s\right)}{1-2\log s}dx\\
	&+(1-\epsilon_{\alpha})\alpha(s)\int_{ B_1}\frac{\left(bQ+c_0q+\psi\right)\log \left(bQ+c_0q+\psi\right)+\left(bQ+c_0q+\psi\right)\log \left(1-2\log s\right)}{1-2\log s}dx\\
	\leq &-\alpha(s)\int_{ \partial B_1}\int_{0}^{1}\frac{r^{\alpha}\left(bQ(\theta)+c_0q(\theta)+\phi(\theta)\right)\log \left(r^{\alpha}\left(bQ(\theta)+c_0q(\theta)+\phi(\theta)\right)\right)}{1-2\log s}\\
	&+\frac{r^{\alpha}\left(bQ(\theta)+c_0q(\theta)+\psi(\theta)\right)\log \left(1-2\log s\right)}{1-2\log s}drd\mathcal{H}^{d-1}\\
	&+(1-\epsilon_{\alpha})\alpha(s)\int_{ \partial B_1}\int_{0}^{1}\frac{r^2\left(bQ(\theta)+c_0q(\theta)+\phi(\theta)\right)\log \left(r^2\left(bQ(\theta)+c_0q(\theta)+\phi(\theta)\right)\right)}{1-2\log s}\\
	&+\frac{r^2\left(bQ(\theta)+c_0q(\theta)+\psi(\theta)\right)\log \left(1-2\log s\right)}{1-2\log s}drd\mathcal{H}^{d-1}\\
	=&-\alpha(s)\int_{ \partial B_1}\int_{0}^{1}\frac{r^{\alpha}\left(bQ(\theta)+c_0q(\theta)+\phi(\theta)\right)\log \left(r^{\alpha}\left(bQ(\theta)+c_0q(\theta)+\phi(\theta)\right)\right)}{1-2\log s}drd\mathcal{H}^{d-1}\\
	&+(1-\epsilon_{\alpha})\alpha(s)\int_{ \partial B_1}\int_{0}^{1}\frac{r^2\left(bQ(\theta)+c_0q(\theta)+\phi(\theta)\right)\log \left(r^2\left(bQ(\theta)+c_0q(\theta)+\phi(\theta)\right)\right)}{1-2\log s}drd\mathcal{H}^{d-1}\\
	=&-\alpha(s)\int_{ \partial B_1}\int_{0}^{1}\frac{r^{\alpha}\left(bQ(\theta)+c_0q(\theta)+\phi(\theta)\right)\alpha\log r}{1-2\log s}drd\mathcal{H}^{d-1}\\
	&+(1-\epsilon_{\alpha})\alpha(s)\int_{ \partial B_1}\int_{0}^{1}\frac{r^2\left(bQ(\theta)+c_0q(\theta)+\phi(\theta)\right)2\log r}{1-2\log s}drd\mathcal{H}^{d-1}\\
	=&-\alpha(s)\int_{0}^{1}\left(\alpha r^{d-1+\alpha}\log r-(1-\epsilon_{\alpha})2r^{d-1+2}\log r\right)\int_{ \partial B_1}\frac{\left(bQ(\theta)+c_0q(\theta)+\phi(\theta)\right)}{1-2\log s}drd\mathcal{H}^{d-1}\\
	=&-\alpha(s)\frac{(2-\alpha)d}{(d+\alpha)^2(d+2)}\int_{ \partial B_1}\frac{bQ(\theta)+c_0q(\theta)+\phi(\theta)}{1-2\log s}d\mathcal{H}^{d-1}.\\
\end{align*}
Therefore, the order of this term in $s$ is $\frac{1}{|\log s|}$.

Furthermore, from the representation of the integral term and Lemma \ref{growth and decay}, which states that for any $0 < s\ll 1$,
\begin{equation*}
	|I(s;u,x^0)| \leq C \frac{\log(-\log s)}{s(\log s)^2}.
\end{equation*}
We can observe from \eqref{inter I} that, on the one hand, the order of this term $\int_{0}^{s}I(\rho;u,x^0)\ d\rho$ in $s$ is $\frac{\log(-\log s)}{|\log s|}$; on the other hand, when $s\ll1$, this integrable term is negative.

Regarding the remaining terms, their orders can be easily identified as higher-order infinitesimals $\left(\frac{1}{|\log s|}\right)$. Consequently, the leading-order term is indeed given by $\int_{0}^{s}I(\rho;u,x^0)\ d\rho$, which is negative. Therefore, the fact that the term $\int_{0}^{s}I(\rho;u,x^0)\ d\rho$ is negative implies that the entire second part remains negative for sufficiently small $s$.

\textbf{Third}, $$\tilde{M}(s;\tilde{\psi})-(1-\epsilon_{\alpha})\tilde{M}(s;\psi),$$
we wish to show that this part is also non-positive. It is worth noting that, unlike in the classical obstacle problem, this part depends on the parameter \(s\). Therefore, we first compute this energy explicitly and express it in terms of \(s\). Then, in order to prove that this part is non-positive, we introduce the following lemma.

\begin{lemma}\label{part 3}
	Let $\psi\in H_0^1(S)$ and consider the $2$-homogeneous extension $\varphi(r,\theta)=r^2\psi(\theta)$ and the $\alpha$-homogeneous extension $\tilde{\varphi}(r,\theta)=r^{\alpha}\psi(\theta)$ to $B_1$, for some $\alpha>2$. Set \begin{align*}
		\epsilon_{\alpha}:=\frac{\alpha-2}{d+\alpha};\qquad \lambda_{\alpha}:=\alpha(\alpha+d-2),
	\end{align*}
	then \begin{align*}
		\tilde{M}(s;\tilde{\varphi})-(1-\epsilon_{\alpha})\tilde{M}(s;\varphi)=\sum_{j=1}^{\infty} c_j^2\frac{\epsilon_{\alpha}}{d+2\alpha-2}(\lambda_{\alpha}+\eta_{\alpha}-\alpha(s)\lambda_j),
	\end{align*}
	where $\eta_{\alpha}:=-\frac{1}{2\log s}\left((\alpha+2)(d+\alpha)-4\right)$.
\end{lemma}

\begin{proof}
	Since \( \|\phi_j\|_{L^2(\partial B_1)} = 1 \) and \( \|\nabla_{\theta} \phi_j\|_{L^2(\partial B_1)} = \lambda_j \), \(j\in N\).
	\[
	\tilde{M}(s;\tilde{ \varphi}) = \alpha(s) \int_{B_1} \frac{1}{2}|\nabla\tilde{\varphi}|^2 \ dx-\int_{ \partial B_1} \tilde{\varphi}^2\ d\mathcal{H}^{d-1}, 	\]
	where\[
	\tilde{\varphi}(r, \theta) = r^{\alpha} \psi(\theta) = r^{\alpha} \sum_{j=1}^{\infty} c_j \phi_j(\theta).
	\]
	Therefore,
	\begin{align*}
		\tilde{M}(s;\tilde{ \varphi})&= \alpha(s) \int_0^1 r^{d-1} dr \int_S |\nabla \tilde{\varphi}(r, \theta)|^2 d\mathcal{H}^{d-1} \, dr - 2\int_S \tilde{\varphi}^2 \, d\mathcal{H}^{d-1} \\
		&= \alpha(s) \int_0^1 r^{d-1} dr \int_S \bigl(\alpha r^{d-1} \sum_{j=1}^{\infty}c_j \phi_j\bigr)^2 + \frac{1}{r^2} \Bigl(r^{\alpha} \sum_{j=1}^{\infty} c_j \nabla_{\theta}\phi_j(\theta)\Bigr)^2 \\
		&\qquad - 2\int_S r^{2\alpha} \left(\sum_{j=1}^{\infty} c_j \phi_j(\theta)\right) \, d\mathcal{H}^{d-1} \\
		&= \sum_{j=1}^{\infty} c_j^2 \left( \alpha(s) \int_0^1 r^{d-1}dr \int_S\left(\alpha^2 r^{2\alpha-2} \phi_j^2+  r^{2\alpha-2} |\nabla_{\theta} \phi_j|^2\right)d\mathcal{H}^{d-1}-2\int_S \phi_j^2 d\mathcal{H}^{d-1}\right)\\
		&= \sum_{j=1}^{\infty} c_j^2 \left( \alpha(s) \int_0^1 r^{d-1+2\alpha-2} (\alpha^2 + \lambda_j) \, dr - 2 \right)\\
		&=\sum_{j=1}^{\infty} c_j^2 \left( \alpha(s)\frac{ \alpha^2+\lambda_j}{d+2\alpha-2} - 2 \right),
	\end{align*}
	when \(\alpha = 2\) and \(\varphi(r, \theta) = r^2 \psi(\theta)\),
	\[
	\tilde{M}(s, \varphi) = \sum_{j=1}^n c_j^2 \left( \alpha(s) \frac{4+\lambda_j}{d+2} - 2 \right).
	\]
	Furthermore, we can compute that
	\begin{align*}
		&\tilde{M}(s, \tilde{ \varphi}) - (1-\epsilon_{\alpha}) \tilde{M}(s, \varphi) \\
		= \ &\sum_{j=1}^{\infty} c_j^2 \left[ \alpha(s) \frac{\alpha^2 + \lambda_j}{d+2\alpha-2} - 2 - (1-\epsilon_\alpha) \left( \frac{4+\lambda_j}{d+2} - 2 \right) \right] \\
		= \ &\sum_{j=1}^{\infty} c_j^2 \Bigg[ \alpha(s) \lambda_j\left( \frac{1}{d+2\alpha-2} - \frac{1-\epsilon_{\alpha}}{d+2} \right) + \alpha(s) \frac{\alpha^2}{d+2\alpha-2} - 2 \\
		&- (1-\epsilon_{\alpha}) \alpha(s) \frac{4}{d+2} + 2(1-\epsilon_{\alpha}) \Bigg]\\
		=\ &\sum_{j=1}^{\infty} c_j^2 \left[ \frac{\epsilon_{\alpha}}{d+2\alpha-2} \left(\lambda_{\alpha}- \alpha(s)\lambda_j -\eta_{\alpha}\right)\right].
	\end{align*}
	Here the last equality follows from direct computations,
	\begin{align*}
		\frac{1}{d+2\alpha-2} - \frac{1-\epsilon_{\alpha}}{d+2}=-\epsilon_{\alpha}\frac{1}{d+2\alpha-2},
	\end{align*}
	and
	\begin{align*}
		\alpha(s) \frac{\alpha^2}{d+2\alpha-2} - 2 - (1-\epsilon_{\alpha}) \alpha(s) \frac{4}{d+2} + 2(1-\epsilon_{\alpha}) 
		= \ \frac{\epsilon_{\alpha}}{d+2\alpha-2} (\lambda_{\alpha}-\eta_{\alpha}), 
	\end{align*}
	where $\eta_{\alpha}=\frac{1}{2\log s}\left((\alpha+2)(d+\alpha)-4\right)$.
	
\end{proof}

\subsection{Homogeneity improvement for the remainder term}\ 

This subsection aims to lay the necessary groundwork for proving the non‑negativity of the third part mentioned above. We hope obtain that 
\begin{align}\label{step 2}
	\tilde{M}(s;\tilde{\psi})-(1-\epsilon_{\alpha})\tilde{M}(s;\psi)\leq \epsilon_{\alpha}^2C_d\sum_{j=1}^ka_j^2-\frac{\epsilon_{\alpha}}{12(d+1)}\|\nabla_{\theta}\phi\|_{L^2(\partial B_1)}^2,
\end{align}

As noted earlier, this part depends on the parameter \(s\). To this end, we exploit the property that when \(s\) is sufficiently small, the function \(-\frac{1}{\log s}\) can be made smaller than some constant, or even tends to \(0+\), in order to estimate the energy of the third part.

Recall that $z$ can be expressed as
\begin{align*}
	z=q_{\nu}+Q_B+\psi,
\end{align*}
where $\psi:=Q_A-Q_B+\varphi$ is a 
$2$-homogeneous function. Therefore, if $A$ is positive definite (i.e., $Q_A\geq 0$), we may choose $Q_B=Q_A$, so that $\sum_{j=1}^{k} a_j^2=0$ and consequently
$\psi:=Q_A-Q_B+\varphi=\varphi$.
Under this choice, the condition that the eigenvalues on the sphere satisfy
$\lambda_j >2d$ implies $\lambda_j \geq 3(d+ 1)$. Hence, inequality \eqref{step 2} holds trivially in this case.

If $Q_A$ changes sign, then we choose
as in \cite[Section 4.2]{csv18}. The goal is to construct a positive definite $Q_B$ while also ensuring that $Q_A-Q_B$
is harmonic. More precisely,
up to a change of coordinates, we may assume that there exist $a_j\geq 0$ for every 
$j=1,\dots,d$ such that
\begin{align*}
	Q_A(x)=-\sum_{j=1}^k a_j x_j^2+\sum_{j=k+1}^d a_j x_j^2,\qquad a_d\geq\frac{1}{2d}\geq \sum_{j=1}^k a_j,
\end{align*}
where the last inequality follows from the fact that $Q_A$ is $L^2(\partial
B_1)$-close to the set of admissible blow‑ups $\mathbb{K}$. We set
\begin{align*}
	Q_B(x):=\sum_{j=k+1}^d a_j x_j^2-\left(\sum_{j=1}^k a_j\right) x_d^2\geq 0,
\end{align*}
where the last inequality—which guarantees the positive definiteness of $B$—depends on $a_d\geq\frac{1}{2d}\geq \sum_{j=1}^k a_j$. Consequently, $Q_A-Q_B$ is an element of the eigenspace $E_{2d}$, corresponding to the eigenvalue $2d$. We choose $\phi_2\in E_{2d}$ and $c_2\in\mathbb{R}$ such that
\begin{align*}
	Q_A-Q_B=c_2\phi_2,\qquad \text{where}\quad\int_{ \partial B_1}\phi_2^2 d\mathcal{H}^{d-1}(\theta)=1.
\end{align*}

Thus, on $\partial B_1$ we can write $\psi$ as
\begin{align*}
	\phi(\theta)=c_2\phi_2(\theta)+\phi(\theta)=c_2\phi_2(\theta)+\sum_{\{j : \lambda_j > 2d\}} c_j\phi_j(\theta).
\end{align*}
Applying Lemma \ref{part 3} we obtain that 
\begin{align*}
	&\tilde{M}(s;\tilde{\psi})-(1-\epsilon_{\alpha})\tilde{M}(s;\phi)\\
	\leq\ &\frac{\epsilon_{\alpha}}{d+2\alpha-2} c_2^2(-\alpha(s)\lambda_2+\lambda_{\alpha}+\eta_{\alpha})+\frac{\epsilon_{\alpha}}{d+2\alpha-2}\sum_{\{j : \lambda_j > 2d\}} c_j^2(-\alpha(s)\lambda_j+\lambda_{\alpha}+\eta_{\alpha}),
\end{align*}
notice that, for small enough $s$, we can obtain that 
\begin{align*}
	&-\alpha(s)\lambda_2+\lambda_{\alpha}+\eta_{\alpha}\leq\ -\lambda_2+\lambda_{\alpha}+\eta_{\alpha}\\
	=\ &(\alpha+d)(\alpha-2)-\frac{1}{2\log s}\left((\alpha+2)(d+\alpha)-4\right)\\
	\leq \ &2(\alpha+d)(\alpha-2)\qquad\qquad\qquad\qquad\text{for}\quad s\ll1.
\end{align*}
Therefore, \begin{align*}
	&\tilde{M}(s;\tilde{\psi})-(1-\epsilon_{\alpha})\tilde{M}(s;\phi)\\
	\leq\ &\frac{2\epsilon_{\alpha}^2(d+\alpha)^2}{d+2\alpha-2} c_2^2+\frac{\epsilon_{\alpha}}{d+2\alpha-2}\sum_{\{j : \lambda_j > 2d\}} c_j^2(-\lambda_j+\lambda_{\alpha}+\eta_{\alpha}),\qquad\text{for}\quad s\ll1.
\end{align*}
Here, 
\begin{align*}
	-\lambda_j+\lambda_{\alpha}+\eta_{\alpha}=\  -\lambda_j\left(1-\frac{\lambda_{\alpha}+\eta_{\alpha}}{\lambda_j}\right)
\end{align*}
and
\begin{align*}
	\lambda_{\alpha}+\eta_{\alpha}=\ &\alpha(\alpha +d-2)-\frac{1}{\log s}\left((\alpha+d)(\alpha+2)-4\right)\\
	\leq\ &\frac{5}{2}\left(d+\frac{1}{2}\right)-\frac{1}{\log s}\left(\left(\frac{5}{2}+d\right)\frac{9}{2}-4\right)\\
	\leq \ & \frac{5}{2}\left(d+\frac{1}{2}\right)-\frac{1}{\log s}\left(\frac{9}{2}d+\frac{9}{4}\right)\\
	\leq\ &\frac{11}{4} \left(d+\frac{1}{2}\right).
\end{align*}
Hence,
\begin{align*}
	1-\frac{\lambda_{\alpha}+\eta_{\alpha}}{\lambda_j}\geq1-\frac{1}{3(d+1)}\frac{11}{4}\left(d+\frac{1}{2}\right)>\frac{d+2}{12(d+1)}>\frac{d+2}{12(d+1)},
\end{align*}
where $\lambda_j\geq 3(d+1)$. In what follows, we choose the constant $\epsilon$ sufficiently small so that
$\alpha\in(2,\frac{5}{2}]$. Thus we obtain 
\begin{align*}
	&\tilde{M}(s;\tilde{\psi})-(1-\epsilon_{\alpha})\tilde{M}(s;\phi)\\
	\leq\ & \frac{2\epsilon_{\alpha}^2(d+3)^2}{d+2\alpha-2} c_2^2-\frac{\epsilon_{\alpha}}{12(d+1)}\sum_{\{j : \lambda_j > 2d\}} \lambda_jc_j^2\\
	=\ &\frac{2\epsilon_{\alpha}^2(d+3)^2}{d+2\alpha-2} \int_{ \partial B_1}(Q_A-Q_B)^2d\mathcal{H}^{d-1}-\frac{\epsilon_{\alpha}}{12(d+1)}\int_{ \partial B_1}|\nabla_{\theta}\phi|^2d\mathcal{H}^{d-1}.
\end{align*}
Then, substituting the expressions of $Q_A$ and $Q_B$ into the above estimate yields \eqref{step 2}.

The next goal is to obtain the higher modes control, i.e.,
\begin{align*}
	\sum_{j=1}^k a_j^2\leq \ C_3\|\nabla_{\theta} \phi\|^{2(1-\gamma)}_{L^2(\partial B_1)},
\end{align*}
where $\gamma=\frac{d-1}{d+3}$.
Since the estimate in this part relies only on the non-negativity of 
$\phi$ and $c$, we omit the details and refer the interested reader to \cite[Section 4.3]{csv18}.

\subsection{Bounding the corrected energy excess by higher modes}\ 

Regarding the proof of the log-epiperimetric inequality, we still need one final inequality to estimate. We \textbf{claim} that there exists $C_4(d)$ such that
\begin{align*}
	M_I(s;z)-\Theta\leq\ C_4(d)\|\nabla_{\theta} \phi\|^{2}_{L^2(\partial B_1)}.
\end{align*}

From the representation \eqref{eq-2.3.1} below, we can see that the functional contains various complicated terms arising from the logarithmic term, including the logarithmic function \(\log s\). Therefore, a more delicate estimate is required at this point. Now, we give the proof of the claim.

\begin{proof}[Proof of the claim]\ 
	From \begin{align*}
		z=q_{\nu}+Q_B+\psi,\quad \psi=r^2\phi,
	\end{align*}
	note that, 
	\begin{align*}
		\int_{ \partial B_1} \phi\  d\mathcal{H}^{d-1}=0.
	\end{align*}
	As the identity \eqref{M-theta}, we know that 
	\begin{align}\label{eq-2.3.1}
		&M_I(s;bQ+c_0q+\psi)-\Theta\\
		=\ &-\alpha(s)\Theta\left((1-b)^2+\frac{1}{2}c_0^2-(1-b)c_0\right)-\frac{1}{2\log s}\Theta -\frac{b^2}{2\log s}\int_{ \partial B_1} Q^2\ dx\nonumber\\
		&-\alpha(s)\int_{ B_1}\frac{\left(bQ+c_0q+\psi\right)\log\left(\left(bQ+c_0q+\psi\right)(1-2\log s)\right)}{1-2\log s} dx\nonumber\\
		&-\frac{c_0^2}{2\log s}\int_{ \partial B_1} q^2 \ d\mathcal{H}^{d-1}-\int_{0}^{\rho}I(\rho;u,x^0)d\rho+\tilde{M}(s;\psi)-\alpha(s)\int_{ B_1}c_0\psi\ dx\nonumber\\
		&-\frac{1}{2\log s}\int_{ \partial B_1} 2c_0q\psi d S +\alpha(s)(1-b)\int_{ B_1}\psi \ dx-\frac{2b}{\log s}\int_{ \partial B_1}Q(c_0q+\psi) d\mathcal{H}^{d-1}.\nonumber
	\end{align}
	Moreover, the estimates for Part 1 and Part 2 yield the following inequality,
	\begin{align*}
		M_I(s;bQ+c_0q+\psi)-\Theta
		\leq\ &-\alpha(s)\frac{c_0^2}{4}\Theta
	+\tilde{M}(s;\psi)-\alpha(s)\int_{ B_1}c_0\psi\ dx	\\
		&-\frac{1}{2\log s}\int_{ \partial B_1} 2c_0q\psi d S +\alpha(s)(1-b)\int_{ B_1}\psi \ dx+o(1),
	\end{align*}
	where $o(1)$ denotes a quantity that tends to $0$ as $s\to 0$ (i.e., an infinitesimal of higher order). Since $\psi(r,\theta)=r^2\phi(\theta)$, it follows that
	\begin{align*}
		\tilde{M}(s;\psi)=\ &\frac{\alpha(s)}{d+2}\int_{ \partial B_1}\left(4\phi^2+|\nabla_{\theta}\phi|^2\right)d\mathcal{H}^{d-1} -2\int_{ \partial B_1} \phi^2 d\mathcal{H}^{d-1}\\
		=\ &\frac{1}{d+2}\int_{ \partial B_1}\left(\alpha(s)|\nabla_{\theta} \phi|^2 +(-2d-4+4\alpha(s))\phi^2\right) d\mathcal{H}^{d-1}.
	\end{align*}
	Due to $q=\frac{1}{2}(x\cdot\nu)_+^2$, it yields that
	\begin{align*}
		&-\alpha(s)\int_{ B_1}c_0\psi\ dx-\frac{1}{2\log s}\int_{ \partial B_1} 2c_0q\psi d\mathcal{H}^{d-1} +\alpha(s)(1-b)\int_{ B_1}\psi \ dx\\
		&-\frac{2b}{\log s}\int_{ \partial B_1}Q(c_0q+\psi)d\mathcal{H}^{d-1}\\
		=\ &\frac{1}{d+2}\left(-\alpha(s)\int_{ \partial B_1} c_0\phi d\mathcal{H}^{d-1} -\frac{d+2}{2\log s}\int_{ \partial B_1^+}c_0\phi d\mathcal{H}^{d-1} +\alpha(s)(1-b)\int_{ \partial B_1} \phi d\mathcal{H}^{d-1}\right)\\
		&-\frac{2b}{\log s}\int_{ \partial B_1}Q(c_0q+\psi)d\mathcal{H}^{d-1}\\
		=\ &-\frac{1}{\log s}\int_{ \partial B_1^+} c_0 \phi d\mathcal{H}^{d-1}-\frac{2b}{\log s}\int_{ \partial B_1}Q(c_0q+\psi)d\mathcal{H}^{d-1},
	\end{align*}
	where the last equality follows from $\int_{ \partial B_1} \phi d \mathcal{H}^{d-1}=0$. Therefore, we obtain
	\begin{align*}
		&M_I(s;bQ+c_0q+\psi)-\Theta\\
		\leq\ &-\frac{c_0^2}{4}\Theta +\frac{1}{d+2}\int_{ \partial B_1}\left(\alpha(s)|\nabla_{\theta} \phi|^2 +(-2d-4+4\alpha(s))\phi^2\right) d\mathcal{H}^{d-1}+o(1)\\
		&-\frac{1}{\log s}\int_{ \partial B_1^+} c_0 \phi d\mathcal{H}^{d-1}-\frac{2b}{\log s}\int_{ \partial B_1}Q(c_0q+\psi)d\mathcal{H}^{d-1}.
	\end{align*}
	Since $c_0$ and $\Theta$ are positive constants, we have that
	\begin{align*}
		M_I(s;bQ+c_0q+\psi)-\Theta&\leq-\frac{c_0}{8}\Theta +\frac{2}{d+2}\int_{ \partial B_1}\left|\nabla_{\theta}\phi\right|^2 d\mathcal{H}^{d-1}\\
		&\leq C_4(d)\int_{ \partial B_1}\left|\nabla_{\theta}\phi\right|^2 d\mathcal{H}^{d-1}.
	\end{align*}
\end{proof}

From Subsections 2.1-2.3, we obtain
\begin{align*}
	\tilde{M}(s;\tilde{\psi})-(1-\epsilon_{\alpha})\tilde{M}(s;\psi)\leq\ & \epsilon_{\alpha}^2C_d\sum_{j=1}^k a_j^2-\frac{\epsilon_{\alpha}}{12(d+1)}\|\nabla_{\theta} \phi\|^2_{L^2(\partial B_1)}\\
	\leq\ & \epsilon_{\alpha}^2C_dC_3\|\nabla_{\theta} \phi\|^{2(1-\gamma)}_{L^2(\partial B_1)}-\frac{\epsilon_{\alpha}}{12(d+1)}\|\nabla_{\theta} \phi\|^2_{L^2(\partial B_1)}.
\end{align*}
Now, choosing $\alpha\in(2,\frac{5}{2}]$ such that
\begin{align*}
	\epsilon_{\alpha}:=\frac{\alpha-2}{d+\alpha}=\epsilon\left(C_4\|\nabla_{\theta}\phi\|^2_{L^2(\partial B_1)}\right)^{\gamma}.
\end{align*}
Hence, we obtain
\begin{align*}
	\tilde{M}(s;\tilde{\psi})-(1-\epsilon_{\alpha})\tilde{M}(s;\psi)\leq\ \epsilon C_4^{\gamma}\left(\epsilon C_2C_3C_4^{\gamma}-\frac{1}{6(d+1)}\right)\|\nabla_{\theta}\phi\|^{2+2\gamma}_{L^2(\partial B_1)}\leq 0,
\end{align*}
for $\epsilon$ small enough.

Consequently, 
\begin{align*}
	M_I(s; v) - \Theta  \leq\ & (M_I(s; z) - \Theta) \left(1 - \epsilon_{\alpha} \right)+ T(s; z)\\
	=\ &(1-\epsilon C_4^{\gamma}\|\nabla_{\theta}\phi\|^{2\gamma}_{L^2(\partial B_1)})(M_I(s; z) - \Theta)+ T(s; z)\\
	\leq\ &\left(1-\epsilon(M_I(s; z) - \Theta)^{\gamma}\right)\left(M_I(s; z) - \Theta\right)+ T(s; z).
\end{align*}
This completes the proof of the log-epiperimetric inequality. 

In the
next section we use it only through the differential inequality it yields for
the corrected Weiss excess. This produces the logarithmic decay of the excess
and the uniqueness of blow-up at singular points.

\section{Decay of the  energy and uniqueness of  blow-up limit}

The goal of this section is to convert the log-epiperimetric inequality
into a quantitative decay estimate for the corrected Weiss excess. In the
classical case the same scheme yields H\"older-type decay; here, because the gain
in the epiperimetric inequality is of logarithmic type, the resulting
convergence rate is logarithmic as well.

\begin{proposition}[Energy decay and uniqueness of blow-up limit]\label{uniqueness}
	Let $x^0\in \mathcal{S}_u$ and suppose that the log-epiperimetric inequality \eqref{eq-log-epi} holds with $\gamma\in(0,1)$ and $r_0 \in (0,1)$ for each $z_r$ defined by
	$$z_r(x):=|x|^2 u_r\left(\frac{x}{|x|}\right)=\frac{|x|^2}{r^2(1-2\log r)} u\left(x^0+\frac{r}{|x|}x\right)$$
	for any $0<r\leq r_0<1$. Assume that $u_0$ denotes an arbitrary blow-up limit of $u$ at $x^0$. Then
	\begin{equation}\label{eq-energy decay}
		\left|W_I(r;u,x^0)-W_I(0+;u,x^0) \right|\leq  \left( -C\gamma\log \dfrac{r}{r_0}\right)^{-\frac{1}{\gamma}},
	\end{equation}
	for $r\in(0,r_0)$, where $W_I(r;u,x^0):=W(r;u,x^0)-\int_{0}^{r}I(\rho;u,x^0)\,d\rho$ and $\gamma\in (0,1)$ is dimensional constant. Moreover, there exists a constant $C(d)>0$ such that
	\begin{equation}\label{eq- unqueness of blow-up}
		\int_{\partial B_1}\left|u_r(x)-u_0(x)\right|\,d\mathcal{H}^{d-1}\leq C\left(-\log \frac{r}{r_0}\right)^{\frac{\gamma-1}{2\gamma}},
	\end{equation}
	for $r\in(0,r_0^2)$, and $u_0$ is the unique blow-up limit of $u$ at $x^0$.
\end{proposition}

\begin{proof}
	Let 
	\begin{align*}
		e(r)=W_I(r;u,x^0)-W_I(0+;u,x^0).
	\end{align*}
	As the proof of \cite[Proposition 1.11]{dz2}, by a direct computation, we have
	\begin{align*}
		e'(r)
		\geq& \frac{d+2}{r}\left(M(r;z_r)+T(r;u_r)-M(r;u_r)\right)\\
		=&\frac{d+2}{r}\Bigg(M(r;z_r)-\int_{0}^{r}I(\rho;u,x^0)d \rho-M_{0}(Q_A)+T(r;z_r)\\
		&-\left(M(r;u_r)-\int_{0}^{r}I(\rho;u,x^0)d \rho-M_{0}(Q_A)\right)\Bigg),
	\end{align*}
	where 
	$$T(r;u_r)=- \frac{2 \alpha(r)}{(d+2)^2 (1-2\log r)} \int_{\partial B_1} z_r \, d\mathcal{H}^{d-1}.$$
	The argument follows the proof of \cite[Section 4]{dz2}, with the only change that the corrected energy \(W_I\) replaces the classical Weiss excess and therefore
	produces logarithmic, rather than H\"older decay.

	Furthermore, due to
	\begin{align*}
		&M(r;z_r)-\int_{0}^{r}I(\rho;u,x^0)d \rho-M_{0}(Q_A)\\
		=:&M_I(r;z_r)-M_{0}(Q_A)\\
		=&\left(M_I(r;z_r)-M_{0}(Q_A)\right)\left(1-\epsilon\left(M_I(r;z_r)-M_{0}(Q_A)\right)^{\gamma}\right)+\epsilon\left(M_I(r;z_r)-M_{0}(Q_A)\right)^{1+\gamma},
	\end{align*}
	where $M_I(r;z_r):=M(r;z_r)-\int_{0}^{r}I(\rho;u,x^0)d \rho$, and applying log-epiperimetric inequality, there exists a radius $r_0>0$ such that for every $r\leq r_0$,  we have that
	\begin{align*}
		e'(r)\geq\ &\frac{d+2}{r}\left(M_I(r;v)-\Theta)-(M_I(r;u_r)-\Theta)\right)+\epsilon\left(M_I(r;z_r)-\Theta\right)^{1+\gamma}\\
		\geq\ &\frac{C}{r}e(r)^{1+\gamma},
	\end{align*}
	where $\Theta:=M_0(Q_A)$ and $\gamma\in (0,1)$ is dimensional constant. Integrating this inequality from $r$ to $r_0$ yields,
	\begin{align*}
		e(r)\leq \left( -C\gamma\log \dfrac{r}{r_0}\right)^{-\frac{1}{\gamma}}.
	\end{align*}
	We note that the energy decay here is logarithmic in $r$, in contrast to the polynomial decay in \cite{dz2}. Hence, the energy decay \eqref{eq-energy decay} is established.
	
	Now, using the energy decay estimate, we obtain the decay estimate for the blow-up limit. By the Cauchy–Schwarz inequality and the Weiss monotonicity formula, we obtain the following inequality:
	for $0<\rho<\sigma< r_0$, we have
	\begin{align*}
		&\int_{\partial B_1} |u_{\sigma}-u_{\rho}|\,d\mathcal{H}^{d-1}\\
		\leq&\int_{\partial B_1}\int_\rho^\sigma\left|\frac{d u_r}{dr}\right|\,dr\,d\mathcal{H}^{d-1}\\
		\leq&\sqrt{d\omega_d}\int_\rho^\sigma r^{-\frac{1}{2}}\sqrt{\frac{dW}{dr}(r;u,x^0)-I(r;u,x^0)}\,dr\\
		= &\sqrt{d \omega_d}\, (\log \sigma - \log \rho)^{\frac12} 
		\Big( e(\sigma) - e(\rho) \Big)^{\frac12}.
	\end{align*}
	For any $0<\rho^{\frac{1}{2}}<r^{\frac{1}{2}}< r_0\ll 1$, there exist integers $l<m$ such that $\frac{\rho}{r_0}\in [2^{-m-1}, 2^{-m})$ and $\frac{r}{r_0} \in [2^{-l-1}, 2^{-l})$. Then it yields that
	\begin{align*}
		&\int_{\partial B_1}\left|{\frac{u(x^0+r x)}{r^2(1-2\log r)}-\frac{u(x^0+\rho x)}{\rho^2(1-2\log \rho)}}\right|d\mathcal{H}^{d-1}\\
		\leq& \int_{ \partial B_1}|u_r-u_{2^{-2^{l+1}}r_0}| d\mathcal{H}^{d-1}+\int_{ \partial B_1}|u_r-u_{2^{-2^{m}}r_0}| d\mathcal{H}^{d-1}\\
		&+\sum_{j=l+1}^{m-1} \int_{\partial B_1}\int_{2^{-2^{j+1}}r_0}^{2^{-2^j}r_0}\left|\frac{d u_r}{dr}\right|\,dr\,d\mathcal{H}^{d-1}\\
		\leq&C \sum_{j=l}^m\left(\log (2^{-2^j})-\log(2^{-2^{j+1}})\right)^{\frac{1}{2}} \left(e(2^{-2^j}r_0)-e(2^{-2^{j+1}}r_0)\right)^{\frac{1}{2}}\\
		\leq&C\left(-\log\frac{r}{r_0}\right)^{\frac{\gamma-1}{2\gamma}}.
	\end{align*}
	Hence, we obtain
	\[
	\int_{\partial B_1}\left|{\frac{u(x^0+r x)}{r^2(1-2\log r)}-\frac{u(x^0+\rho x)}{\rho^2(1-2\log \rho)}}\right|d\mathcal{H}^{d-1}\leq 
	C\left(-\log\frac{r}{r_0}\right)^{\frac{\gamma-1}{2\gamma}},
	\]
	where $\gamma$ is the exponent in the log-epiperimetric inequality.
	Consequently, for any sequence $\rho_j \to 0+$, the rescaled solutions 
	\[
	\frac{u(x^0+\rho_j x)}{\rho_j^2 (1-2\log \rho_j)} \to u_0 \quad \text{in } L^1(\partial B_1),
	\]
	and this establishes \eqref{eq- unqueness of blow-up} and completes the proof.
\end{proof}

\section{The structure of the singular  set of the free boundary}

In this section, we apply Proposition \ref{uniqueness} uniformly on compact subsets of the
singular set. The key point is to show that, for singular points in a fixed
compact set, one can choose a common small scale at which both the corrected
energy excess and the distance to the cone \(\mathcal{K}\) are sufficiently small. Once
this is achieved, the uniqueness of the blow-up follows from Proposition \ref{uniqueness},
and the geometric description of the strata is then obtained by the standard
Whitney extension argument.

\begin{proposition}\label{prop 4.1}
	Let $\Omega\subset \mathbb{R}^d$ be an open set and $u\in H^1(\Omega)$ a minimizer of \eqref{functional}. Then for every compact set $\Omega_0\Subset\Omega$, there is $r_0>0$ and a constant $C=C(d,\Omega_0,\Omega)>0$, such that for every free boundary point $x_0\in \mathcal{S}_u\cap \Omega_0$ and any $r\in(0,r_0)$, we have that
	\begin{align*}\label{eq- unqueness of blow-up 2}
		\int_{\partial B_1}\left|\frac{u(x^0+rx)}{r^2(1-2\log r)}-\frac{1}{2}x\cdot A(x^0)x\right|\,d\mathcal{H}^{d-1}\leq C\left(-\log r\right)^{\frac{\gamma-1}{2\gamma}}.
	\end{align*}
\end{proposition}

\begin{remark}
	Proposition \ref{prop 4.1} is the compactness input needed for the geometric description of the singular strata. It upgrades the pointwise convergence of singular
	blow-ups to a uniform small-scale statement on compact subsets of \(S_u\), which is precisely the form required to apply the log-epiperimetric
	inequality uniformly.
\end{remark}

\begin{proof}[Proof of Proposition \ref{prop 4.1}]
	The proof is divided into four parts.
	
	(i) For $x\in S_u\cap \Omega_0$ and
	$0<r<\operatorname{dist}(\Omega_0,\partial\Omega)$, set
	\[
	e(x,r):=W_I(r;u,x)-\Theta .
	\]
	Since $x\in S_u$, we have
	\[
	\lim_{r\rightarrow 0+} e(x,r)=0.
	\]
	We claim that for every $\varepsilon>0$ there exists $r_\varepsilon>0$ such that
	\[
	e(x,r)\le \varepsilon
	\qquad\text{for every }x\in S_u\cap \Omega_0,\;\; 0<r\le r_\varepsilon .
	\]
	Assume by contradiction that this fails. Then there exist $\varepsilon_0>0$,
	a sequence $x_j\in S_u\cap \Omega_0$, and radii $r_j\rightarrow 0+$ such that
	\[
	e(x_j,r_j)>\varepsilon_0
	\qquad\text{for every }j.
	\]
	Since $S_u\cap \Omega_0$ is compact, passing to a subsequence if necessary,
	we may assume that
	\[
	x_j\to \bar x\in S_u\cap \Omega_0.
	\]
	By the monotonicity of $r\mapsto W_I(r;u,x_j)$, for any fixed $\rho>0$ and all
	$j$ sufficiently large such that $r_j<\rho$, we have
	\[
	\varepsilon_0
	< W_I(r_j;u,x_j)-\Theta
	\le W_I(\rho;u,x_j)-\Theta.
	\]
	Hence
	\[
	\varepsilon_0
	<
	\bigl(W_I(\rho;u,x_j)-W_I(\rho;u,\bar x)\bigr)
	+\bigl(W_I(\rho;u,\bar x)-\Theta\bigr).
	\]
	For fixed $\rho>0$, the map $x\mapsto W_I(\rho;u,x)$ is continuous; moreover,
	since $\bar x\in S_u$, we have
	\[
	W_I(\rho;u,\bar x)\rightarrow \Theta
	\qquad\text{as }\rho\rightarrow 0+.
	\]
	Therefore, the right-hand side can be made arbitrarily small by first choosing
	$\rho$ sufficiently small and then $j$ sufficiently large, which is impossible.
	This proves the claim.

	(ii) Let $\rho_j \to 0+$ and $x^j \in\mathcal{S}_u\cap \Omega_0$, and define the rescaled functions
	\[
	u_j(x) := \frac{u(x^j + \rho_j x)}{\rho_j^2 (1 - 2 \log \rho_j)}.
	\] 
	Assume that $u_j \to v$ in $W^{1,2}_{\mathrm{loc}}(\mathbb{R}^d)$ as $j \to \infty$. Recalling that the blow-up limits $u_0$ satisfy $\Delta u_0 = \chi_{\{u_0 > 0\}}$, we deduce that $v$ is a homogeneous global solution of degree $2$ to the classical obstacle problem, and its energy satisfies
	\[
	\begin{aligned}
		M_0(v) 
		&= \lim_{j \to \infty} \left[ \alpha(\rho_j) \int_{B_1} \left(\frac{1}{2} |\nabla u_j(y)|^2 + G(\rho_j,u_j(y))\right) \, dy - \int_{\partial B_1} |u_j(y)|^2 \, d\mathcal{H}^{d-1} \right] \\
		&= \lim_{j \to \infty} W(\rho_j; u, x^j) = \lim_{j \to \infty} W_I(\rho_j; u, x^j) = \Theta.
	\end{aligned}
	\]
	By Caffarelli’s dichotomy theorem, the blow-up limit $u_0$	can only take one of two forms: it belongs either to $\mathbb{H}$ or to $\mathbb{K}$, with the corresponding energy density $M_0(u_0)$ being either $\frac{\Theta}{2}$ or $\Theta$, respectively. Conversely, if the energy density 
	$M_0(u_0)$ of a blow-up limit is $\Theta$, then $u_0\in\mathbb{K}$. Hence, we conclude that $v \in \mathbb{K}$.
	
	(iii) We claim that for any sufficiently small $\rho>0$, the rescaled function
	\[
	\frac{u(\bar{x}+\rho x)}{\rho^2(1-2 \log \rho)},
	\]
	is uniformly close to $\mathbb{K}$ in the $W^{1,2}_{\mathrm{loc}}(\mathbb{R}^d)$-topology for all $\bar{x} \in\mathcal{S}_u\cap \Omega_0$.
	
	To verify this claim, we proceed by contradiction. Suppose the claim fails. Then there exist sequences $\rho_j \to 0+$ and $x^j \in\mathcal{S}_u\cap \Omega_0$ such that, for any $Q_A \in \mathbb{K}$,
	\begin{equation}\label{eq5.1}
		\left\| \frac{u(x^j+\rho_j x)}{\rho_j^2 (1-2\log \rho_j)} - Q_A \right\|_{W^{1,2}(B_1)} \geq \delta > 0.
	\end{equation}
	By the growth estimates in \cite[Lemma 3.10 and Theorem 3.11]{qs17}, we have
	\[
	|u_j| + |\nabla u_j| \leq C,
	\]
	hence  $\|u_j\|_{L^{\infty}_{\mathrm{loc}}(\mathbb{R}^d)} \leq C$, which implies
	\[
	\|u_j\|_{W^{1,2}_{\mathrm{loc}}(\mathbb{R}^d)} \leq C.
	\] 
	Consequently, there exists a subsequence, still denoted by $u_j$, such that
	\[
	u_j \to u_0 \quad \text{in } W^{1,2}_{\mathrm{loc}}(\mathbb{R}^d).
	\] 
	By statement (ii), the limit $u_0 \in \mathbb{K}$, which contradicts \eqref{eq5.1}. This verifies the claim.
	
	(iv) Finally, we complete the proof of Proposition \ref{prop 4.1}.
	
	Indeed, the claim in (iii) ensures that the assumptions of Proposition \ref{uniqueness} hold. Therefore, there exists a constant $C=C(d,r_0,\Omega_0,\Omega)>0$ such that, for any $x^0 \in\mathcal{S}_u\cap \Omega_0$ and $r \in (0,r_0)$,
	\[
	\begin{aligned}
		\int_{\partial B_1} \left| \frac{u(x^0+rx)}{r^2(1-2\log r)} - u_0(x) \right| d\mathcal{H}^{d-1} 
		&\leq C_1 \left(-\log \frac{r}{r_0}\right)^{-\frac{1-\gamma}{2\gamma}} \\
		&\leq C \left(-\log r\right)^{-\frac{1-\gamma}{2\gamma}},
	\end{aligned}
	\]
	where $u_0$ is the unique blow-up limit of $u$ at $x^0$, and $u_0 \in \mathbb{K}$.
\end{proof}

Therefore, the above proposition tells us that at any singular point of the free boundary, the blow-up limit is unique and converges at a logarithmic rate, thereby enabling the study of the geometric structure of singular points of the free boundary.

\begin{proof}[Proof of Theorem \ref{regularity of FB}]
	We know that
	\begin{align*}
		\left\|Q_{A(x_1)}-Q_{A(x_2)}\right\|_{L^2(\partial B_1)}\leq c(n)\int_{ \partial B_1}\left|Q_{A(x_1)}(x)-Q_{A(x_2)}(x)\right|d\mathcal{H}^{d-1}.
	\end{align*}
	Using the triangular inequality, we have
	\begin{align*}
		& \int_{\partial B_1} \left|Q_{A(x_1)}(x)-Q_{A(x_2)}(x)\right| d\mathcal{H}^{d-1} \\
		\leq & \int_{\partial B_1} \Bigg| Q_{A(x_1)}(x)- \frac{u(x_1 + r x)}{r^2(1-2\log r)} \\
		&\qquad + \frac{u(x_1 + r x)}{r^2(1-2\log r)} - \frac{u(x_2 + r x)}{r^2(1-2\log r)} + \frac{u(x_2 + r x)}{r^2(1-2\log r)} - Q_{A(x_2)}(x) \Bigg| d\mathcal{H}^{d-1} \\
		\leq & 2 C (-\log r)^{\frac{\gamma-1}{2\gamma}} + \int_{\partial B_1} \left| \frac{u(x_1 + r x)}{r^2(1-2\log r)} - \frac{u(x_2 + r x)}{r^2(1-2\log r)} \right| d\mathcal{H}^{d-1} \\
		\leq & 2 C  (-\log r)^{\frac{\gamma-1}{2\gamma}} + \int_{\partial B_1} \int_0^1 \frac{|\nabla u(x_1 + r x + t (x_2-x_1))|}{r^2 (1-2\log r)} |x_1-x_2| \, dt \, d\mathcal{H}^{d-1}.
	\end{align*}
	Based on this, we revisit the optimal Log-Lipschitz regularity for the gradient of a minimizer of \eqref{functional} established by de Queiroz and Shahgholian in \cite[Theorem 3.11]{qs17}, namely, let $u$ be a minimizer of \eqref{functional} and $\Omega_0\subset\subset \Omega$, there exist $r_0>0$ and a constant $C>0$ depending both only on $\text{dist}(\Omega_0,\partial \Omega)$, $\varphi$ and $d$ such that, if $x\in \Omega_0$ with $d(x)=\text{dist}(x,\partial \{u>0\})\leq r_0$, then
	\begin{equation*}
		|\nabla u(x)|\leq Cd(x)\log \frac{1}{d(x)}.
	\end{equation*}
	Thus, for sufficiently small $r_0$, it follows that
	\begin{align*}
		&\int_{\partial B_1} \left| \frac{u(x_1 + r x)}{r^2(1-2\log r)} - \frac{u(x_2 + r x)}{r^2(1-2\log r)} \right| d\mathcal{H}^{d-1} \\
		\leq &  \int_{\partial B_1} \int_0^1 \frac{|\nabla u(x_1 + r x + t (x_2-x_1))|}{r^2 (1-2\log r)} |x_1-x_2| \, dt \, d\mathcal{H}^{d-1}\\
		\leq&\int_{\partial B_1}\int_{0}^{1} \frac{Cd(x_1 + r x + t (x_2-x_1))\log \frac{1}{d(x_1 + r x + t (x_2-x_1))}}{r^2(1-2\log r)}|x_2-x_1|dtd\mathcal{H}^{d-1}\\
		\leq&\int_{\partial B_1}\int_{0}^{1} \frac{C\max\{r;|x_2-x_1|\}\log \frac{1}{\max\{r;|x_2-x_1|\}}}{r^2(1-2\log r)}|x_2-x_1|dtd\mathcal{H}^{d-1}\\
		\leq&C_1\frac{r\log \frac{1}{r}}{r^2(1-2\log r)}|x_2-x_1|\\
		=&C_1\frac{-\log r |x_1-x_2|}{r(1-2\log r)}\\
		=& C_1\frac{-\log\left(|x_1-x_2|\left(-\log |x_1-x_2|\right)^{\frac{1-\gamma}{2\gamma}}\right)|x_1-x_2|}{|x_1-x_2|\left(-\log |x_1-x_2|\right)^{\frac{1-\gamma}{2\gamma}}\left(1-2\log\left(|x_1-x_2|\left(-\log |x_1-x_2|\right)^{\frac{1-\gamma}{2\gamma}}\right)\right)}\\
		\leq & C_2 \frac{-\log\left(|x_1-x_2|\left(-\log |x_1-x_2|\right)^{\frac{1-\gamma}{2\gamma}}\right)}{\left(-\log |x_1-x_2|\right)^{\frac{1-\gamma}{2\gamma}}\left(-2\log\left(|x_1-x_2|\left(-\log |x_1-x_2|\right)^{\frac{1-\gamma}{2\gamma}}\right)\right)}\\
		= &C_2\left(-\frac{1}{\log |x_1-x_2|}\right)^{\frac{1-\gamma}{2\gamma}},
	\end{align*}
	where we choose $r=|x_1-x_2|\left(-\log |x_1-x_2|\right)^{\frac{1-\gamma}{2\gamma}}$. Therefore, we obtain that	\begin{align*}
		& \int_{\partial B_1} \left|Q_{A(x_1)}(x)-Q_{A(x_2)}(x)\right| d\mathcal{H}^{d-1} \\
		\leq & 2 C  \left(-\log \left(|x_1-x_2|\left(-\log |x_1-x_2|\right)^{\frac{1-\gamma}{2\gamma}}\right)\right)^{\frac{\gamma-1}{2\gamma}} + C_2\left(-\frac{1}{\log |x_1-x_2|}\right)^{\frac{1-\gamma}{2\gamma}}\\
		\leq& C\left(-\log |x_1-x_2|\right)^{-\frac{1-\gamma}{2\gamma}}.
	\end{align*}
	Therefore, in summary, we obtain \eqref{eq-regular of Fb} in Theorem \ref{regularity of FB}. Subsequently, by a series of standard methods such as the Whitney extension theorem, we can obtain the geometric structure of the singular set of the free boundary. For the detailed standard proof, we refer to \cite{FGS15,psu} and omit it here.
\end{proof}

\
\
\

{\bf Author Contributions.}~ All the three authors have equally contributed to the research work in and the writing of this article.

{\bf Funding.}~ This work was supported by National Nature Science Foundation of China (Grants 12125102, 12526202, 124B2012), Nature Science
Foundation of Guangdong Province \\(2024A1515012794), Shenzhen Science and Technology Program (JCYJ20241202124209011),
Postdoctoral Fellowship Program and China Postdoctoral Science Foundation (BX20250058), and China Postdoctoral Science Foundation (2025M783083).

{\bf Data Availability.}~ No datasets were generated or analysed during the current study.

{\bf Declarations}

{\bf Competing Interests.}~ The authors declare no competing interests.

\end{document}